\documentclass[11pt, letterpaper]{amsart}
\usepackage{graphicx} 

\usepackage{amsthm,amsmath,amssymb,amsfonts,latexsym,graphicx,enumerate,setspace,verbatim,rotating}

\usepackage{color}                       
\usepackage[dvipsnames]{xcolor}
\definecolor{verde}{HTML}{288906}
\usepackage{hyperref}      

\usepackage{enumitem}
\usepackage{caption}
\usepackage{subcaption}
\usepackage{tikz-cd}
\usepackage{xfrac}
\usepackage{mathtools}

\usepackage{tikz}\usetikzlibrary{backgrounds}
\usetikzlibrary {shapes.geometric, arrows, arrows.meta}
\pgfdeclarelayer{bg}    
\pgfsetlayers{bg,main}
\usetikzlibrary{decorations.pathreplacing,calligraphy}
\usetikzlibrary{tikzmark}

\usepackage[utf8]{inputenc}
\usepackage{pgfplots}
\pgfplotsset{compat=1.11}
\usepgfplotslibrary{fillbetween}
\usetikzlibrary{intersections}
\pgfdeclarelayer{bg}
\pgfsetlayers{bg,main}
\usetikzlibrary{calc}
\definecolor{tempcolor}{RGB}{224,235,239}
\usepackage[utf8]{inputenc}
\usepackage{dirtytalk}
\usepackage{graphicx}
\usepackage{float}
\usepackage{multicol}

\usepackage{blindtext}
\def\multiset#1#2{\ensuremath{\left(\kern-.3em\left(\genfrac{}{}{0pt}{}{#1}{#2}\right)\kern-.3em\right)}}

\theoremstyle{plain} 
\newtheorem{theorem}{Theorem}[section]
\newtheorem{lemma}[theorem]{Lemma}
\newtheorem{proposition}[theorem]{Proposition}
\newtheorem{corollary}[theorem]{Corollary}
\newtheorem{conjecture}[theorem]{Conjecture}

\newtheorem{thm}[]{Theorem}

\theoremstyle{definition}  
\newtheorem{defn}[theorem]{Definition}
\newtheorem{example}[theorem]{Example}

\theoremstyle{remark}  
\newtheorem*{remark}{Remark}

\DeclareMathOperator{\depth}{depth }
\DeclareMathOperator{\height}{height }
\DeclareMathOperator{\hgt}{ht }
\DeclareMathOperator{\pmd}{pmd }
\DeclareMathOperator{\Cl}{Cl }
\DeclareMathOperator{\cl}{cl }
\DeclareMathOperator{\Min}{Min }
\DeclareMathOperator{\Ass}{Ass }
\DeclareMathOperator{\adj}{adj }
\DeclareMathOperator{\grade}{grade }
\DeclareMathOperator{\chara}{char}
\DeclareMathOperator{\ini}{in}

\renewcommand\appendix[1]{
\chapter*{#1}
\addcontentsline{toc}{chapter}{#1}
}

\title[Normality, factoriality and strong $F$-regularity of LSS rings]{Normality, factoriality and strong $F$-regularity of Lovász-Saks-Schrijver rings}
\author{Eliana Tolosa Villarreal}
\address{Dipartimento di Matematica, Dipartimento di Eccellenza 2023-2027,
Universit\`a degli Studi di Genova, Italy}
\email{etolosav@gmail.com}
\thanks{The author was supported  by the MIUR Excellence Department Project awarded to the Dept.~of Mathematics, Univ.~of Genova, CUP D33C23001110001,  and by GNSAGA-INdAM.\\
This material is based upon work supported by the National Science Foundation under Grant No. DMS-1928930, while the author was in residence at the Mathematical Sciences Research Institute in Berkeley, California, during the semester/year of Spring 2024"}
\date{}
\begin{document}

\begin{abstract}
    Every simple finite graph $G$ has an associated Lovász-Saks-Schrijver ring $R_G(d)$ that is related to the $d$-dimensional orthogonal representations of $G$. The study of $R_G(d)$ lies at the intersection between algebraic geometry, commutative algebra and combinatorics. We find a link between algebraic properties such as normality, factoriality and strong $F$-regularity of $R_G(d)$ and combinatorial invariants of the graph $G$. In particular we prove that if $d \geq \pmd(G)+k(G)$ then $R_G(d)$ is $F$-regular in finite characteristic and rational singularity in characteristic $0$ and furthermore if $d \geq \pmd(G)+k(G)+1$ then $R_G(d)$ is UFD. Here $\pmd(G)$ is the positive matching decomposition number of $G$ and $k(G)$ is its degeneracy number. 
\end{abstract}

\maketitle

\section{Introduction}
  Lovász introduced in 1979 \cite{Lovasz79} the orthogonal representations of a graph $G$ as a tool to investigate certain invariants of $G$, in particular the so called Shannon capacity. An orthogonal representation associates to each vertex of $G$ a vector in a (real) vector space $W$ so that non-adjacent vertices are associated to orthogonal vectors. An orthogonal representation is said to be in general position if every subset of vertices of $G$ of size at most $\dim W$ is sent to linearly independent vectors of $W$. 

The set of all orthogonal representations of  $G$ in a given vector space give rise to an algebraic variety whose defining equations express the orthogonality conditions. Furthermore, the (closure of the) set of  the general position orthogonal representations is  a subvariety of  it (possibly empty). 
Indeed Lovász, Saks and Schrijver studied in \cite{LSS} the irreducibility and non-emptiness of the variety of general position orthogonal representations of $G$. Lovász's  recent book \cite{LovaszBook} revises these developments and  discusses connections with other subjects as, for example,  quantum physics. 
 
 The equations defining  orthogonal representations of $G$ give rise to the Lovázs-Saks-Shrijver ideals and rings (LSS for short), named after the authors. 
 
 The study of such LSS varieties, ideals and rings  establishes a new connection between   commutative algebra and combinatorics and, as discussed in \cite{Conca_2019},  it is also  intimately connected with classical subjects such as invariant theory and determinantal rings. 
 
Let us recall the formal definition of LSS ideals and rings in the generality we need. We will follow the conventions of \cite{Conca_2019} according to which the points on the variety associated to the LSS ideal of $G$ correspond to the orthogonal representation of the complementary graph.  
  
\begin{defn} Let $\mathbb{K}$ be a field and $d$ a positive integer. 
    Let $G = ([n], E)$ be a simple graph with $[n] = \{1, \dots, n\}$.  Consider the polynomial ring on $nd$ variables $S = \mathbb{K}[y_{ij} \hspace{2pt} | \hspace{2pt} i \in [n], \hspace{1pt} j \in [d] \hspace{1pt}]$. For each edge $e = \{i,j\} \in E$ we set 
    $$f_{e} = \sum \limits_{k=1}^{d} y_{ik}y_{jk}.$$
    The Lovázs-Saks-Shrijver ideal $L_{G}(d)$ of the graph $G$ is the ideal generated by all such polynomials, that is,
    $$L_{G}(d) = (f_e \hspace{2pt} | \hspace{2pt} e \in E) \subseteq S.$$
    We call the quotient ring $R_G(d) = \Large\sfrac{S}{L_G(d)}$ the Lovázs-Saks-Shrijver ring of the graph $G$, LSS ring for short. 
\end{defn}

 The systematic study of algebraic properties of LSS ideals and rings started with the papers of Herzog, Macchia, Saeedi Madani and Welker \cite{HerzogEtAll2015} for $\dim W=2$ and Conca and Welker \cite{Conca_2019} for arbitrary $\dim  W$. For further results on LSS ideals and rings see also \cite{GSW} and \cite{Kumar}.
 
In  \cite{HerzogEtAll2015}  the authors proved  that when  $\chara (\mathbb{K}) \neq 2$, $L_G(2)$ is always a radical ideal while  for $\chara (\mathbb{K}) = 2$, the ideal $L_G(2)$ is radical if $G$ is a bipartite graph. 

In \cite{Conca_2019} the authors  presented various examples of  non-radical ideals $L_G(d)$ with $d>2$ in characteristic $0$. Furthermore, they introduced a new combinatorial invariant of $G$, the \textit{positive matching decomposition number} or $\pmd(G)$, and proved that for $d \geq \pmd(G)$ the ideal $L_G(d)$ is a radical complete intersection and for $d \geq \pmd(G) + 1$ it is prime. Denote by $\Delta(G)$ the degree of $G$ and by $k(G)$ its degeneracy number and set $\alpha(G) = \Delta(G) + k(G) - 1$ (see Section \ref{GraphT} for definitions). In his 2019 preprint \cite{Kapon} Kapon shows that for $d\geq \alpha(G)$ the ideal $L_G(d)$ is a complete intersection and for $d\geq \alpha(G) +1$ it defines an irreducible variety. These results suggest that there might be a general relationship between the two purely combinatorial  invariants $\pmd(G)$ and $\alpha(G)$. For certain families of graphs (e.g.~trees and complete graphs) they coincide but there are also graphs for which we have $\pmd(G) < \alpha(G)$ (see the graph in Example \ref{ExPMD-Termorder}). We do not know of examples where the opposite inequality holds. This makes us believe that for any graph $G$ we have $\pmd(G) \leq \alpha(G)$.

 We will focus on stronger algebraic properties of LSS rings such as normality, factoriality and $F$-regularity. Our results involve again the degeneracy $k(G)$ and the positive matching decomposition number $\pmd(G)$ of $G$. We list our main results next.

\begin{thm}[Theorem \ref{ThmSFR}]
\label{ThmA}

    Let $G = ([n], E)$ be a simple graph and let $d \geq \pmd(G)+ k(G)$. 
    \begin{itemize}
  \item[(1)] If  $\chara \mathbb{K} = p>0$ then $R_G(d)$ is strongly $F$-regular,
  \item[(2)] If  $\chara \mathbb{K} = 0$ then $R_G(d)$ has rational singularities.
  \end{itemize}
\end{thm}

\begin{thm}[Theorem \ref{ThmUFD}]
\label{ThmB}
    Let $G = ([n], E)$ be a simple graph. If $d \geq \pmd(G) + k(G) + 1$, then the LSS ring $R_G(d)$ is \textit{UFD}.
\end{thm}

Notice that in the special case where $G$ is a forest we have that $\pmd(G) = \Delta(G)$ and $k(G)=1$. Hence Theorem \ref{ThmB} asserts that if $d \geq \Delta(G) + 2$ the LSS ring $R_G(d)$ is a UFD. By virtue of \cite[Thm. 1.5]{Conca_2019} we know that the ring $R_G(d)$ is a domain if an only if $d \geq \Delta(G)+1$. The next theorem discusses this case.

\begin{thm}[Theorem \ref{ThmNormal}]
\label{ThmC}
    Let $G = (V, E)$ be a forest and $d = \Delta(G)+1$. Then, the LSS ring $R_G(d)$ is normal.
\end{thm}

Furthermore, we determine  the divisor class group of two families of trees: the \textit{star} (Proposition \ref{ThmDCGStar}) and \textit{path} (Proposition \ref{ThmDCGPath}) graphs. In particular, these examples show that Theorem \ref{ThmB}  is sharp at least for forests. This study, together with some experimental data led us to formulate the next conjecture:

\begin{conjecture}
    Let $G$ be a forest  and $d = \Delta(G) +1$. The divisor class group of $R_G(d)$ is
    $$\Cl(R_G(d)) \cong \mathbb{Z}^m,$$
    where $m = |\{v \in V : \delta(v) = \Delta(G)\}|$, with $\delta(v)$ the degree of the vertex $v$ in $G$.
\end{conjecture}

We give an overview on how the paper is organized. In Section \ref{SectNotation} we will introduce general notation and some   generalities. In Section \ref{Sec2} we will focus on the strong $F$-regularity property, concluding with the proof of Theorem \ref{ThmA}. In Section \ref{Sec3} we will discuss factoriality and prove Theorem \ref{ThmB}. The last section (\ref{Sec4}) will be dedicated to proving Theorem \ref{ThmC} and computing the divisor class groups for the \textit{star} and \textit{path} graphs.

\vspace{1cm}

\textbf{Acknowledgments:} The author would like to thank Aldo Conca for his continued support and guidance throughout this project

\section{Notations and preliminaries}
\label{SectNotation}
In this section we will establish the notation used throughout the paper as well as some known results on graphs, Gr\"obner basis and LSS ideals that are relevant for further discussion. We will also state and prove a theorem regarding a special localization of the LSS ring.

\subsection{Graph theory}
\label{GraphT}
We will always consider $G = (V,E)$ to be a simple graph on a finite vertex set $V$ which will usually be assumed to be $V = [n] = \{1,\dots, n\}$. We will denote by $\delta(i)$ the degree of the vertex $i$ and by $\Delta(G)$ the degree of the graph $G$. A subgraph $G' = (V', E')$ of $G$ is a graph such that $V' \subseteq V$ and $E' \subseteq E$. The subgraph induced by a set of vertices $U\subseteq V$ is the subgraph $H = (U, \{\{i,j\} \in E \hspace{2pt} | \hspace{2pt} i,j \in U\})$. Let $U\subset V$, we denote by $G\backslash U$ the induced subgraph of $G$ on $V\backslash U$. We denote with $G_i$ the subgraph of $G$ obtained by eliminating from $G$ the $i$-th vertex and all of its adjacent edges, that is the same as $G_i = G\backslash \{i\}$.

We will denote $K_n$ the complete graph on $n$ vertices, that is $K_n = ([n], \{\{i,j\} \hspace{2pt} | \hspace{2pt} 1\leq i < j \leq n\})$ and $K_{m,n}$ the complete bipartite graph, which is the graph $K_{m,n} = ([m]\cup[\tilde n], \{\{i,\tilde j\} \hspace{2pt} | \hspace{2pt} i\in [m] \text{ and } j \in [\tilde n]\})$.

We recall the definition of degeneracy of a graph as it is an invariant that will be useful in future sections. A graph $G$ is said to be $k$-degenerate, for $k$ a non-negative integer, if every induced subgraph of $G$ has a vertex of degree at most $k$. Notice that all graphs are $\Delta(G)$-degenerate. The degeneracy of the graph $G$, denoted $k(G)$, is the smallest value $k$ for which the graph is $k$-degenerate. It is possible to calculate algorithmically $k(G)$ for any graph $G$ by repeatedly removing the vertex with smaller degree. The degeneracy is given by the highest degree of any vertex at the time of its removal.

\subsection{Gr\"obner Bases Theory}
 In this subsection we will describe the basics of Gr\"obner bases theory as well as mention some useful results. For further information on this topic the reader can look at \cite{BWCRV-GB} and \cite{KR}.
 Consider the polynomial ring $S = \mathbb{K}[x_1, \dots, x_n]$ and let $x^\alpha = x_1^{\alpha_1} x_2^{\alpha_2}\dots x_n^{\alpha_n}$ be a monomial in $S$. A term order $\prec$ is a total order on the set of monomials $x^\alpha$ such that
 \begin{enumerate}
     \item $1 \prec x^{\alpha}$,
     \item if $x^{\alpha} \prec x^{\beta}$ then $x^{\alpha +\gamma} \prec x^{\beta +\gamma}$ for all 
     $x^{\gamma}$.
 \end{enumerate}
Let $f$ be a non-zero polynomial in $S$, we denote by $\ini_{\prec}(f)$ its largest monomial under the term order $\prec$. Let $I = (f_1, \dots, f_t)$ be an ideal of $S$, we define its initial ideal under $\prec$ as the ideal
$$\ini_{\prec}(I) := ( \ini_{\prec}(f) \hspace{2pt} | \hspace{2pt} f \in I\backslash\{0\}).$$
A Gr\"obner basis $\mathcal{G}$ of $I$ with respect to $\prec$ is a set of polynomials $g_1, \dots, g_s \in I$ such that
$$\ini_{\prec}(I) = (\ini_{\prec}(g_1), \dots, \ini_{\prec}(g_s)).$$

It is also possible to introduce Gr\"obner bases theory using a weight vector instead of a term order. This way of defining Gr\"obner bases actually generalizes the term order setting. Let $\omega = (\omega_1, \dots, \omega_n) \in \mathbb{R}^n$ be a weight vector and let 
$$ f = \sum \limits_{\alpha \in \mathbb{N}^n} a_{\alpha} x^{\alpha}$$
be a non-zero polynomial in $S$. We define 
$$m_{\omega}(f) = \max \limits_{a_\alpha \neq 0} \{\alpha \cdot \omega\}.$$
 
The initial form of $f$ with respect to $\omega$ is the following polynomial
$$\ini_{\omega}(f) = \sum\limits_{\alpha \cdot \omega = m_{\omega}(f)} a_{\alpha} x^{\alpha},$$
and if $I$ is an ideal, the initial ideal with respect to $\omega$ is 
$$\ini_{\omega}(I) = ( \ini_{\omega}(f) \hspace{2pt} | \hspace{2pt} f \in I\backslash\{0\}).$$
Note that $\ini_{\omega}(I)$ is not, in general, a monomial ideal but it is homogeneous with respect to the graded structure induced on $S$ by $\omega$.

A Gr\"obner basis with respect to a weight vector $\omega$ is defined analogously to the term order definition.
It is known that for every ideal $I$ and for every term order $\prec$ there exists a weight vector $\omega$ such that $\ini_{\prec}(I) = \ini_{\omega}(I)$ \cite[Lem.1.5.6.]{BWCRV-GB}. Then, we will state all results in terms of $\omega$ but they are also true for a term order $\prec$. We present some properties of the initial ideals next.

\begin{proposition}\cite[Prop.1.6.2.]{BWCRV-GB}
\label{GBProp}
Let $S = \mathbb{K}[x_1, \dots, x_n]$ be the polynomial ring on a field $\mathbb{K}$, $I$ be an homogeneous ideal of $S$ and let $\omega \in \mathbb{R}^n$ be a weight vector. We have the following assertions:
\begin{enumerate}
    \item If $\ini_{\omega}(I)$ is radical, then $I$ is radical,
    \item If $\ini_{\omega}(I)$ is complete intersection, then $I$ is complete intersection,
    \item If $\ini_{\omega}(I)$ is prime, then $I$ is prime.
\end{enumerate}
\end{proposition}

\begin{proposition}\cite[Prop.1.5.2.]{BWCRV-GB}
\label{GBPropRS2}    
    Let $S = \mathbb{K}[x_1, \dots, x_n]$ be the polynomial ring on a field $\mathbb{K}$, $I$ an homogeneous ideal of $S$ and $g_1, \dots g_c \in S$ homogeneous polynomials. Let $\omega$ be a weight vector. If $\ini_\omega(g_1), \ini_\omega(g_2), \dots, \ini_\omega(g_c)$ is a regular sequence in $\large \sfrac{S}{\ini_\omega(I)}$, then $g_1, \dots, g_c$ is a regular sequence on $S/I$.
    Moreover, 
    $$\ini_\omega(I + (g_1, \dots, g_c)) = \ini_\omega(I) + (\ini_\omega(g_1), \dots, \ini_\omega(g_c))$$
\end{proposition}

\subsection{Known results on LSS ideals}
In this subsection we will present some known results on  LSS  ideals. To that end we give a formal definition of the positive matching decomposition number of a graph $G$ and show its connection to algebraic properties of the associated  LSS  ideal. Furthermore we exhibit how these algebraic properties are also linked to the degeneracy of $G$.

The next theorem explains the stabilization of certain algebraic properties as $d$ increases as well as the closure of such properties under taking subgraphs of a graph $G$.

\begin{theorem}\cite{Conca_2019}
\label{CWStable}
  Let $G = (V,E)$ be a graph. Then,
  \begin{enumerate}
      \item \label{part1}If $L_G(d)$ is prime then $L_G(d)$ is a complete intersection.
      \item \label{part2}If $L_G(d)$ is a complete intersection then $L_G(d+1)$ is prime.
      \item If $L_G(d)$ is prime (respectively complete intersection) then $L_{G'}(d)$ is prime (respectively complete intersection) for every subgraph $G'$ of $G$.
  \end{enumerate}
\end{theorem}
Parts \ref{part1} and \ref{part2} imply that if $L_G(d_0)$ is prime (or complete intersection) for some $d_0$ we have the property for any $d\geq d_0$. Given $G$, a natural question to ask, then, what is the smallest  $d$ (if any)  such that those properties do hold.  To this end we need to first give some definitions.

\begin{defn}
    Let $G = (V,E)$ be a graph. A positive matching of $G$ is a subset $M \subset E$ of pairwise disjoint edges of $G$ such that there exists a weight function $w \hspace{1pt} : \hspace{1pt} V \rightarrow \mathbb{R}$ satisfying:

    \begin{align*}
        \sum \limits_{i \in e} w(i) > 0 &\text{ if } e\in M,\\
        \sum_{i \in e} w(i) < 0 &\text{ if } e\in E\backslash M.
    \end{align*}

    A positive matching decomposition of $G$ is a partition of the set of edges $E = \cup_{i=1}^p E_i$ into pairwise disjoint subsets such that $E_i$ is a positive matching on the graph $(V, E\char`\\ \cup_{j=1}^{i-1} E_j)$ for $i=1, \dots , p.$
  
  The smallest $p$ for which $G$ admits a positive matching decomposition is called the positive matching decomposition number and is denoted by $\pmd(G)$.
\end{defn}

The next lemma summarizes some properties of the positive matching decomposition number.

\begin{lemma}\cite{Conca_2019}
    Let $G = ([n], E)$ be a graph, then
    \begin{enumerate}
        \item $\pmd(G) \leq \text{min}\{2n-3, |E| \},$
        \item If $G$ is bipartite then $\pmd(G) \leq \min \{n-1, |E|\}$,
        \item $\pmd(G) \geq \Delta(G)$ with equality  when $G$ is a forest.
    \end{enumerate}
\end{lemma}

Based on Conca and Welker's work, Farrokhi, Gharakhloo and Yazdan Pour \cite{FGYP} studied the positive matching decomposition number from a graph theoretical perspective. In their work, they compute the $\pmd$ for certain families of graphs and discuss the complexity of such computations. 
Moreover, Farrokhi wrote a computer program to compute the $\pmd$ of a graph $G$ \cite{F}. 

We now introduce a result that allow us to put together the positive matching decomposition number of a graph and the existence of a special term order that will be of some importance for proving algebraic properties of the  LSS  ideals.

\begin{lemma}
\label{LemTermOrd}
    Let $G = (V,E)$ be a graph, $d \geq \pmd(G) = p$ and $E = \cup_{l=1}^p E_l$ a positive matching decomposition. Then, there exists a term order $\prec$ such that for every component $E_l$ and every edge $\{i,j\}\in E_l$ we have
    $$\ini_{\prec} (f_{ij}) = y_{il}y_{jl}.$$
\end{lemma}

\begin{remark}
    Notice that as $E_l$ is a matching for each $l = 1, \dots, p$, the previous lemma implies that the initial forms of the generators $f_{ij}$ of $L_G(d)$ are pairwise coprime and square free monomials.
\end{remark}

The proof of Lemma \ref{LemTermOrd} is constructive. This construction is shown in the following example, but for a more accurate definition the reader is referred to \cite{Conca_2019}.

\begin{example}
\label{ExPMD-Termorder}
    Let $G = (V,E)$ be the following graph,
    
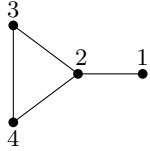
\begin{center}
        \begin{tikzpicture}[scale= 0.43]
        \node at (0,0)[circle,fill,inner sep=1.3pt]{};
        \node at (-2,1.5)[circle,fill,inner sep=1.3pt]{};
        \node at (-2,-1.5)[circle,fill,inner sep=1.3pt]{};
        \node at (2,0)[circle,fill,inner sep=1.3pt]{};
        \draw (0,0) -- (-2,1.5) -- (-2,-1.5) -- (0,0) -- (2,0);
        \node[] at (0.1,0.5) {\footnotesize{2}};
        \node[] at (2,0.5) {\footnotesize{1}};
        \node[] at (-2,2) {\footnotesize{3}};
        \node[] at (-2,-2) {\footnotesize{4}};  
    \end{tikzpicture}
    \end{center}
    We exhibit the construction of a positive matching decomposition for $G$, with weight functions $w_1, w_2, w_3$ shown in blue, in the following figure.
    
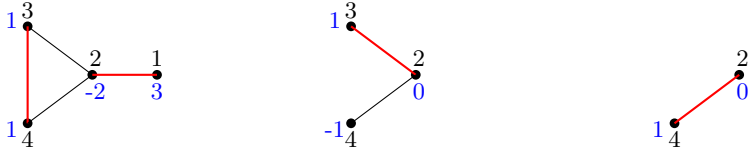
\begin{center}
    \begin{tikzpicture}[scale= 0.43]

        \node at (0,0)[circle,fill,inner sep=1.3pt]{};
        \node at (-2,1.5)[circle,fill,inner sep=1.3pt]{};
        \node at (-2,-1.5)[circle,fill,inner sep=1.3pt]{};
        \node at (2,0)[circle,fill,inner sep=1.3pt]{};
        \draw (0,0) -- (-2,1.5);
        \draw [thick, red] (-2,1.5)-- (-2,-1.5);
        \draw [] (-2,-1.5)-- (0,0);
        \draw [thick, red] (0,0)-- (2,0);
        \node[] at (0.1,0.5) {\footnotesize{2}};
        \node[] at (2,0.5) {\footnotesize{1}};
        \node[] at (-2,2) {\footnotesize{3}};
        \node[] at (-2,-2) {\footnotesize{4}}; 
        \node[] at (0.1,-0.52) {\footnotesize{\color{blue}-2}};
        \node[] at (2,-0.52) {\footnotesize{\color{blue}3}};
        \node[] at (-2.5,1.7) {\footnotesize{\color{blue}1}};
        \node[] at (-2.5,-1.7) {\footnotesize{\color{blue}1}}; 

        \node at (10,0)[circle,fill,inner sep=1.3pt]{};
        \node at (8,1.5)[circle,fill,inner sep=1.3pt]{};
        \node at (8,-1.5)[circle,fill,inner sep=1.3pt]{};
        \draw [thick, red] (10,0) -- (8,1.5);
        \draw [] (8,-1.5) -- (10,0);
        \node[] at (10.1,0.5) {\footnotesize{2}};
        \node[] at (8,2) {\footnotesize{3}};
        \node[] at (8,-2) {\footnotesize{4}}; 
        \node[] at (10.1,-0.52) {\footnotesize{\color{blue}0}};
        \node[] at (7.5,1.7) {\footnotesize{\color{blue}1}};
        \node[] at (7.5,-1.7) {\footnotesize{\color{blue}-1}}; 

        \node at (20,0)[circle,fill,inner sep=1.3pt]{};
        \node at (18,-1.5)[circle,fill,inner sep=1.3pt]{};
        \draw [thick, red] (18,-1.5) -- (20,0);
        \node[] at (20.1,0.5) {\footnotesize{2}};
        \node[] at (18,-2) {\footnotesize{4}}; 
        \node[] at (20.1,-0.52) {\footnotesize{\color{blue}0}};
        \node[] at (17.5,-1.7) {\footnotesize{\color{blue}1}}; 
    
    \end{tikzpicture}
    \end{center}
    If $E_1=\{12,34\}$, $E_2=\{23\}$ , $E_3=\{24\}$, then $E= E_1 \cup E_2 \cup E_3$ is a positive matching decomposition of $G$. Notice that this decomposition is minimal, hence $\pmd(G) = 3$.
    Let $d \geq 3$ and $S = \mathbb{K}[y_{ij} \hspace{2pt} | \hspace{2pt} i \in [4], j \in [d]]$ equipped with the term order $\prec$ described in Lemma \ref{LemTermOrd}. The LSS ideal $L_G(d)$ is generated by the following polynomials
    \begin{align*}
    f_{12} &= \underline{y_{11}y_{21}}+y_{12}y_{22}+y_{13}y_{23}+\dots + y_{1d}y_{2d},\\
    f_{23} &= y_{21}y_{31}+\underline{y_{22}y_{32}}+y_{23}y_{33}+\dots + y_{2d}y_{3d},\\
    f_{24} &= y_{21}y_{41}+y_{22}y_{42}+\underline{y_{23}y_{43}}+\dots + y_{2d}y_{4d},\\
    f_{34} &= \underline{y_{31}y_{41}}+y_{32}y_{42}+y_{33}y_{43}+\dots + y_{3d}y_{4d},
    \end{align*}
    where we underlined the leading terms under $\prec$. Notice that $\ini_{\prec}(f_{12})$, $\ini_{\prec}(f_{23})$, $\ini_{\prec}(f_{34})$, $\ini_{\prec}(f_{24})$ are pairwise coprime for $d \geq 3 = \pmd(G)$.

\end{example}

Next theorem shows a precise connection between algebraic properties of $L_G(d)$ and the $\pmd$ of its associated graph $G$.

\begin{theorem}\cite{Conca_2019}
\label{CWThm}
    For any simple graph $G =(V,E)$ we have the following.
    \begin{enumerate}
        \item If $d \geq \pmd(G)$ then the ideal $L_G(d)$ is a radical complete intersection.
        \item If $d \geq \pmd(G) +1$ then the ideal $L_G(d)$ is prime.
    \end{enumerate}
\end{theorem}

There exist graphs for which $L_G(d)$ being a radical complete intersection ideal and a prime ideal is reached for the same value of $d$. In particular, the cycle graph with $6$ vertices, $C_6$, it is not either prime nor complete intersection for $d = 2$ and it is a prime complete intersection ideal when $d=3$.

Independently, Kapon found a different bound for $d$ such that the ideal $L_G(d)$ is complete intersection and its radical is prime \cite{Kapon}. We recall his result next.



\begin{theorem}\cite{Kapon}
\label{Kapon}
    Let $G = (V, E)$ be a graph with degree $\Delta(G)$ and degeneracy $k(G)$. Set $\alpha(G) = \Delta(G) + k(G)-1$. We have the next assertions:
    \begin{enumerate}
        \item If $d \geq \alpha(G) $ then $L_G(d)$ a is complete intersection.
        \item If $d \geq \alpha(G) +1$ then the variety defined by $L_G(d)$ is irreducible.
    \end{enumerate}
\end{theorem}

 Theorems \ref{CWThm} and \ref{Kapon} give an answer for the stabilization point of $L_G(d)$ being a prime and complete intersection ideal. Notice that for some families of graphs, such as forests, complete graphs and cycles the exposed bounds coincide. Nonetheless, there are numerous examples of graphs for which the $\pmd$ bound is better. In particular, consider $G$ to be the graph in Example \ref{ExPMD-Termorder}. Its positive matching decomposition number is $3$ while $\alpha(G) = 4$.

\begin{remark}
    Example \ref{ExPMD-Termorder} can be generalized to the construction of a family of graphs for which the difference between the two numbers, $\alpha$ and $\pmd$, is arbitrarily large. Consider the graph $G$ to be the complete graph $K_n$ on $n$ vertices with $n-2$ edges (and their respective vertices) glued to one of its vertices. Notice that $\Delta(G) = \Delta(K_n) + n -2 = 2n-3$ and $k(G) = k(K_n) = n-1$. Therefore, $\alpha (G) = 3n-5$. In the other side we have from \cite[Thm.2.3.]{FGYP} that $\pmd(G) = \max \{\pmd(K_n), \Delta(G)\} = 2n-3$. In this case $\alpha(G)-\pmd(G) = n-2$. Hence the difference between the positive matching decomposition number and $\alpha$ is unbounded.
\end{remark}

 This remark together with other computations lead us to formulate the following conjecture.
 \begin{conjecture}
     Let $G$ be a simple graph, then $\pmd(G) \leq \alpha(G)$.
 \end{conjecture}

\subsection{Algebraic background}
Some classical results in commutative algebra are presented in this subsection as they will be useful for proving the main results of this paper.

We start by recalling two well known criteria: one for normality and the other one for factoriality proved by Serre and Nagata respectively. For a complete proof of these theorems see \cite{fossum2012divisor} and \cite{Nagata1957}.

\begin{theorem}[Serre's criterion for normality]
\label{SerreNormal}
    The domain $R$ is normal if and only if the following conditions hold:
\begin{enumerate}
    \item[(R$_1$)] The ring $R_P$ is regular for all height 1 prime ideals $P$ of $R$.
    \item[(S$_2$)] $\depth R_P \geq \min \{2, \height P\}$ for all prime ideals $P$ of $R$.
\end{enumerate}
\end{theorem}

\begin{theorem}[Nagata's criterion for factoriality]
\label{Nagata}
    Let $R$ be a noetherian domain and let $x\in R$. If $(x)\subseteq R$ is prime and $R_x$ is \textit{UFD}, then $R$ is \textit{UFD}.
\end{theorem}
This result is a special case of Nagata's Theorem described in \cite{Nagata1957}. To apply these criteria to our setting we have to deal with certain localizations of $R_G(d)$. We identify an element $D \in R_G(d)$ for which the localization of $R_G(d)$ at $D$ allows us to make conclusions about algebraic properties. We start by setting up some notation.

Let $G = (V,E)$ be a graph with $n$ vertices and $c$ edges and let $G' = (V', E')$ be the graph coming from $G$ by eliminating one vertex $n$ and its $t$ adjacent edges. Note that $|V'| = n-1$ and $|E'| = c-t$. We label from $1$ to $t$ the vertices in $G$ (and in $G'$) that are adjacent to $n$.
In the following, the variables associated to vertex $n$ play a special role and hence we rename them by setting $x_i := y_{ni}$ for $i = 1, \dots, d$. Notice that $S = S'[x_1, \dots, x_d]$. Also, $L_G(d) = L_{G'}(d) + (f_{1n}, \dots, f_{tn})$, where
$$f_{in} = \sum \limits_{j=1} ^{d} x_jy_{ij}.$$

Let $A$ be the generic matrix $t\times t$ that uses the variables $y_{ij}$ for $1 \leq i \leq t$ and $ d-t+1 \leq j \leq d$, that is,

\begin{equation}
\label{A}
   A =
     \begin{pmatrix}
      {y_{1,d-t+1}} & {y_{1, d-t+2}} & \dots  & {y_{1,d-1}} & {y_{1,d}}\\
      {y_{2,d-t+1}} & {y_{2, d-t+2}} & \dots  & {y_{2,d-1}} & {y_{2,d}}\\
      \vdots & & \ddots & & \\
      {y_{t,d-t+1}} & {y_{t, d-t+2}} & \dots  & {y_{t,d-1}} & {y_{t,d}}\\
     \end{pmatrix}. 
\end{equation}

Moreover, let
\begin{equation}
\label{D}
    D = \det(A).
\end{equation}

Sometimes we will consider the matrix $A$ and its determinant inside the quotient rings $R_G(d)$ or $R_{G'}(d)$, depending on the case. We will maintain the same notation for elements and their respective classes in order to simplify the notation, but we will emphasize the ring we are working on when needed.
We state some Lemmas that will be useful for proving the main results.

\begin{lemma}
\label{LemmaDnonZero}
    Let $G$, $G'$, $L_G(d)$, $R_G(d)$ and $D$ defined as above. Then,  $D \neq 0$ in $R_G(d)$.
\end{lemma}

\begin{proof}
We will show that $D \notin L_G(d)$. Let $Q = \left( y_{ik}y_{jk} \hspace{2pt} | \hspace{2pt} 1\leq i \leq n, \hspace{1pt} 1\leq j\leq n, \hspace{1pt} 1\leq k \leq d \right)$. Notice that $L_G(d) \subseteq Q$. But, all monomials of $D$ are of the form $y_{i_1j_1}y_{i_2j_2}\dots y_{i_tj_t}$ with $j_r\neq j_s$ for all $r,s \in \{1, \dots, t\}$. Then $D \notin Q$ and therefore $D \notin L_G(d)$. 
\end{proof}

\begin{lemma}
\label{LemmaInclusion}
    Let $G$, $G'$, $R_G(d)$, $R_{G'}(d)$ defined as above. Then, the canonical map $R_{G'}(d) \rightarrow R_G(d)$ is an inclusion.
\end{lemma}

\begin{proof}
    Consider the natural map
    \begin{align*}
        \phi: S' &\rightarrow R_G(d)\\
        y_{ij} &\mapsto \overline{y_{ij}}
    \end{align*}
    We would like to show that $ker \phi = S' \cap L_G(d)$ is equal to $L_{G'}(d)$.
    Notice that as $L_{G'}(d) \subseteq L_G(d)$ and $L_{G'}(d)$ is an ideal of $S'$ we have that $L_{G'}(d) \subseteq ker \phi$. On the other hand, consider the multigraded structure on $S$ induced by $\deg y_{ij} = \mathfrak{e}_i \in \mathbb{Z}^n$. Let $f \in S' \cap L_G(d)$ be a polynomial with $\deg f = (\alpha_1, \alpha_2, \dots, \alpha_n) \in \mathbb{Z}^n$. Then, as $f \in S'$ we have that $\alpha_n = 0$. As $f \in L_G(d) = L_{G'}(d)+ (f_{1n}, \dots, f_{tn})$ and $\deg f = (\alpha_1, \alpha_2, \dots, 0)$ we can conclude that $f \in L_{G'}(d)$. As this is valid for any $f \in S'\cap L_G(d)$ we have that $ker \phi = L_{G'}(d)$ which means that it exists an inclusion $R_{G'}(d) \xhookrightarrow{} R_G(d)$ sending each class $\overline{y_{ij}}\in R_{G'}(d)$ to its respective class in $R_G(d)$.
\end{proof}

\begin{lemma}
\label{RDcong}
Let $G$, $G'$, $R_G(d)$, $R_{G'}(d)$ and $D$ defined as above. Assume $L_G(d)$ is a complete intersection (e.g. $d \geq \pmd(G)$) then,
$$R_G(d)_D \cong R_{G'}(d) [x_{1}, \dots, x_{d-t}]_D.$$
\end{lemma}

\begin{proof}
   It is enough to show
    \begin{enumerate}[label=(\alph*)]
         \item $ R_{G'}(d) [\overline{x_{1}}, \dots, \overline{x_{d-t}}]_D =  R_G(d)_D,$
         \item The elements $\overline{x_{1}}, \dots, \overline{x_{d-t}}$ are algebraically independent over $R_{G'}(d)$.
    \end{enumerate}
    To show part (a) we will prove both containments. Let us start by showing $R_{G'}(d)[\overline{x_1}, \dots, \overline{x_{d-t}}]_D \subseteq R_G(d)_D$. By Lemma \ref{RDcong} we know that $R_{G'}(d) \subseteq R_G(d)$.
    Then, as $S'[x_1, \dots, x_{d-t}]\subseteq S$ we have that
    $$R_{G'}(d)[\overline{x_1}, \dots, \overline{x_{d-t}}]_D \subseteq R_G(d)_D.$$
    To prove $R_G(d)_D \subseteq R_{G'}(d)[\overline{x_1}, \dots, \overline{x_{d-t}}]_D$ it is enough to show that the elements $\overline{x_{d-t+1}}, \dots, \overline{x_d}$ are in $R_{G'}(d)[\overline{x_1}, \dots, \overline{x_{d-t}}]_D$. In $R_G(d)$ we have
    \begin{align*}
            f_{1n} &=  y_{11}x_1 + y_{12}x_2 + \dots + y_{1d}x_d = 0, \\ 
            f_{2n} &=  y_{21}x_1 + y_{22}x_2 + \dots + y_{2d}x_d = 0, \\
            \vdots \\
            f_{tn} &= y_{t1}x_1 + y_{t2}x_2 + \dots + y_{td}x_d = 0.
    \end{align*}
    By moving the first $d-t$ terms of each polynomial to the right hand side we obtain
    \begin{align*}
            x_{d-t+1}y_{1,d-t+1} + x_{d-t+2}y_{1,d-t+2} + \dots + x_{d}y_{1d} &= -\sum \limits_{i=1}^{d-t} x_i y_{1i} \\
           x_{d-t+1}y_{2,d-t+1} + x_{d-t+2}y_{2,d-t+2} + \dots + x_{d}y_{2d} &= -\sum \limits_{i=1}^{d-t} x_i y_{2i} \\
           & \vdots\\
            x_{d-t+1}y_{t,d-t+1} + x_{d-t+2}y_{t,d-t+2} + \dots + x_{d}y_{td} &= -\sum \limits_{i=1}^{d-t} x_iy_{ti}.\\
    \end{align*}
    We can write it in matrix form, that is 
    \[
         A
        \begin{pmatrix}
         x_{d-t+1}\\ x_{d-t+2}\\ \vdots\\ x_{d}
        \end{pmatrix}
        = - \begin{pmatrix}
        \sum \limits_{i=1}^{d-t} x_i y_{1i}\\
        \sum \limits_{i=1}^{d-t} x_i y_{2i}\\
        \vdots\\
        \sum \limits_{i=1}^{d-t} x_iy_{ti}
        \end{pmatrix}
        \]

        where $A$ is the matrix defined in \ref{A}.
        Recall that $D = \det(A)$, then, multiplying both sides by the $\adj(A)$ we have
        \begin{equation*}
            D  
        \begin{pmatrix}
         x_{d-t+1}\\ x_{d-t+2}\\ \vdots\\ x_{d}
        \end{pmatrix}
         = B,
        \end{equation*}
        with $B$ a matrix with entries in the ring $S'[x_1, \dots, x_{d-t}]$.
        Therefore, in the localization $R_G(d)_D$, we can write each variable $x_{d-t+1}, \dots, x_d$ in terms of the rest of the variables in $S'[x_{1}, \dots, x_{d-t}]$, that is $\overline{x_{d-t+1}}, \dots, \overline{x_d} \in R_{G'}(d)[\overline{x_1}, \dots, \overline{x_{d-t}}]_D$.
        Then $R_G(d)_D = R_{G'}(d)[\overline{x_1}, \dots, \overline{x_{d-t}}]_D$ and we have proved part (a).

        To prove part (b) notice that
        \begin{equation}
        \label{eq2}
            \dim R_{G'}(d) [\overline{x_1}, \dots, \overline{x_{d-t}}] \leq \dim R_{G'}(d) + d-t,
        \end{equation}
        with the equality if and only if $\overline{x_1}, \dots, \overline{x_{d-t}}$ are algebraically independent. From part (a) we deduce that $R_G(d)$ and $R_{G'}(d)[\overline{x_1}, \dots, \overline{x_{d-t}}]$ have the same field of fractions, therefore the same transcendental degree over $\mathbb{K}$. This implies that 
         $$\dim R_{G'}(d)[\overline{x_1}, \dots, \overline{x_{d-t}}] = \dim R_G(d).$$
        Then, as $d \geq \pmd(G)$ the ideals $L_G(d)$ and $L_{G'}(d)$ are complete intersections by \ref{CWThm} and \ref{CWStable}. Then, we have $\dim R_{G'}(d) = \# \text{ variables in } S' - \#E' = (n-1)d - \#E'$ and $\dim R_{G}(d) = nd - \#E$.
        Therefore, in equation (\ref{eq2}) we have
        $$nd - \#E \leq (n-1)d - \#E' + d-t,$$
        which is an equality and we can conclude that $\overline{x_1}, \dots, \overline{x_{d-t}}$ are algebraically independent over $R_{G'}(d)$.
\end{proof}

\section{Strong $F$-regularity}
\label{Sec2}

The goal for this section is to prove Theorem \ref{ThmA}. In order to do that we first recall the Glassbrenner criterion for complete intersections which stands in the core of our proof.

We use $I^{[p^e]}$ to denote the $p^e$ Frobenius power of the ideal $I \subseteq R$, that is,
$$I^{[p^e]} = (a^{p^e} \hspace{2pt} | \hspace{2pt} a\in I).$$

\begin{theorem} [\cite{Glassbrenner}]
\label{CoroGlassbrenner}
    Let $S = \mathbb{K}[x_1, \dots, x_n]$ where $\mathbb{K}$ is a field of prime characteristic $p>0$. Let $\mathfrak{M} = (x_1, \dots, x_n)$, $I = (f_1, \dots, f_t)$ an ideal generated by homogeneous elements and $R = S/I$. Suppose $s \in S$ is a homogeneous polynomial such that $s \notin P$ for all $P \in Min(I)$ and $R_s$ is strongly $F$-regular.
    Then, $R$ is strongly $F$-regular if and only if there exists $e>0$ such that 
    $$s(f_1^{p^e-1} \cdot \ldots \cdot f_t^{p^e-1})\notin \mathfrak{M}^{[p^e]}.$$
\end{theorem}

Now we can present the proof of Theorem \ref{ThmA}.

\begin{theorem}
\label{ThmSFR}
    Let $G = ([n], E)$ be a simple graph and let $d \geq \pmd(G)+ k(G)$. 
    \begin{itemize}
  \item[(1)] If  $\chara \mathbb{K} = p>0$ then $R_G(d)$ is strongly $F$-regular,
  \item[(2)] If  $\chara \mathbb{K} = 0$ then $R_G(d)$ has rational singularities.
    \end{itemize}
\end{theorem}

\begin{proof} By virtue of \cite{Smith, MehtasSrinivas} (2) follows from (1).
To prove (1) we use induction on $n$, the number of vertices of $G$. For $n=1$ and $d \geq 1$, $R_G(d) = S = \mathbb{K}[y_1, \dots, y_d]$ is the polynomial ring on $d$ variables. Then, $R_G(d)$ is a regular ring, hence $R_G(d)$ is strongly $F$-regular. For $n > 1$ take $G'$, $S'$ and $R_{G'}(d)$ to be as described at the end of Section \ref{SectNotation}. By induction hypothesis $R_{G'}(d)$ is strongly $F$-regular for $d \geq \pmd(G') + k(G') +1$. We have to show that $R_G(d)$ is strongly $F$-regular for $d \geq \pmd(G) + k(G) +1 \geq \pmd(G') + k(G') +1$. To do so consider the element $D \in R_G(d)$ as described in \ref{D} and recall that by Lemma \ref{LemmaDnonZero} $D \neq 0$ in $R_G(d)$. By virtue of Theorem \ref{CoroGlassbrenner}, to prove strong $F$-regularity of $R_G(d)$ it is enough to show:
    \begin{enumerate}
        \item[(1)] \label{part1}$R_G(d)_D$ is strongly $F$-regular,
        \item[(2)] There exists $e>0$ such that $D \left( \prod\limits_{\{i,j\} \in E} f_{ij}^{p^e-1}\right) \notin \mathfrak{M}^{[p^e]}.$
    \end{enumerate}

From Lemma \ref{RDcong} we have that $R_G(d)_D \cong R_{G'}(d) [x_{1}, \dots, x_{d-t}]_D$. Strong $F$-regularity preserves under localization and polynomial extensions, then as $R_{G'}(d)$ was strongly $F$-regular by induction hypothesis, we can conclude that $R_{G}(d)_D$ is strongly $F$-regular proving part (1).

For part (2) we will show that there exists a term order $\prec$ such that 
$$ \ini_{\prec} \left( D \left( \prod\limits_{\{i,j\} \in E} f_{ij}^{p^e-1}\right)\right) \notin \mathfrak{M}^{[p^e]}.$$
Consider the term order used in Lemma \ref{LemTermOrd}. As all the $\ini_{\prec}(f_{ij})$ are pairwise coprime and square-free for $\{i,j\} \in E$, we have that in 
$$\prod \limits_{\{i,j\} \in E} \ini_{\prec}(f_{ij}) ^{p^e-1}$$
each variable appears with power at most $p^e -1$.
Notice that $\ini_{\prec}(f_{ij})$ uses the variables with index $j \leq \pmd(G)$ for all $\{i,j\} \in E$
and $D$ uses the variables with index $j > \pmd(G)$. Then, as $d \geq \pmd(G) + k(G)$ we have that  $D$ uses variables not in $\ini_{\prec}(f_{ij})$ for all $\{i,j\} \in E$. Hence, in the product 
$$\ini_{\prec}(D) \left( \prod \limits_{\{i,j\} \in E} \ini_{\prec}(f_{ij}) ^{p^e-1} \right)$$
each variable appears with exponent at most $p^e-1$.
We can then conclude that
$$D \left( \prod\limits_{\{i,j\} \in E} f_{ij}^{p^e-1}\right) \notin \mathfrak{M}^{[p^e]}$$
for any $e>0$.
\end{proof}

\section{Unique Factorization Domain}
\label{Sec3}

In this section we prove that there is a stabilization point for $R_G(d)$ being a unique factorization domain. This stabilization point depends on the combinatorics of the graph $G$, more specifically, it is related to the positive matching decomposition number and the degeneracy of $G$. Nonetheless, the given bound is probably not sharp, therefore it is possible that for a specific graph the UFD stabilization is achieved earlier. We prove the main result of the section, Theorem \ref{ThmB}, next.

\begin{theorem}
\label{ThmUFD}
    Let $G = ([n], E)$ be a simple graph. If $d \geq \pmd(G) + k(G) + 1$, then $R_G(d)$ is \textit{UFD}.
\end{theorem}

\begin{proof}
We use induction on $n$, the number of vertices of the graph $G$. For $n=1$ and $d\geq 1$ we have that $R_G(d) = S = \mathbb{K}[y_1, \dots, y_d]$ and it is a unique factorization domain. Let $n > 1$ and take $G'$, $S'$ and $R_{G'}(d)$ as described at the end of Section \ref{SectNotation}. By induction hypothesis $R_{G'}(d)$ is a unique factorization domain for $d\geq \pmd(G')+ k(G') +1$, we will prove that $R_G(d)$ is a UFD for $d\geq \pmd(G)+k(G)+1 \geq \pmd(G')+ k(G') +1$.
To do so we use Nagata's criterion for factoriality (\ref{Nagata}). From Theorem \ref{CWThm} we know that $R_G(d)$ is a Noetherian domain as $d\geq \pmd(G) + 1$. Consider the element $D\in R_G(d)$ defined in \ref{D}. Then, by virtue of Theorem \ref{Nagata}, it is enough to show
\begin{enumerate}
      \item \label{PPrime}$(D) \subseteq R_G(d)$ is a prime ideal,
      \item \label{PUFD} $R_G(d)_D$ is a unique factorization domain.
  \end{enumerate}

Proving part \ref{PPrime} is the same as proving that $L_G(d) + (D)$ is prime on $S$. To show this, by Proposition \ref{GBProp}, it is enough to prove that for certain weight vector $w$, the initial ideal $\ini_w(L_G(d) + (D))$ is prime. Take $w$ defined as follows:
    $$w(y_{ij}) = \begin{cases}
		2, & \text{for $j \leq \pmd(G) +1$}\\
        1, & \text{for $j > \pmd(G) +1$}.
	\end{cases}$$

As $D$ is a non-zero divisor of $R_G(d)$, by Proposition \ref{GBPropRS2} we have that
\begin{equation}
\label{EqIn}
    \ini_{\omega} (L_G(d) + (D)) = \ini_{\omega}(L_G(d)) + \ini_{\omega}((D)).
\end{equation}

Notice that $\ini_w((D)) = (D)$ as $D$ only uses variables $y_{ij}$ such that $j>\pmd(G)+1$ and all of them have the same weight. Let $f_{ij}^d$ be the generator of $L_G(d)$ associated to the edge $\{i,j\}$. Notice that $\ini_w(f_{ij}^d) = f_{ij}^{\hat{d}}$ with $\hat{d} = \pmd(G)+1$. By Theorem \ref{CWThm} $L_G(d)$ is a complete intersection, then the elements of the form $f_{ij}^d$ such that $\{i,j\} \in E$ form a regular sequence. Using again Proposition \ref{GBPropRS2} we have
    $$\ini_{\omega}(L_G(d)) = (\ini_{\omega} (f_{ij}^d) \hspace{2pt} | \hspace{2pt} \{i,j\} \in E) = (f_{ij}^{\hat{d}} \hspace{2pt} | \hspace{2pt} \{i,j\} \in E) = L_G(\hat{d}).$$
    
Therefore equation \ref{EqIn} translates to
\begin{equation*}
    \ini_{\omega} (L_G(d) + (D)) = L_G(\hat{d}) + (D).
\end{equation*}

    Again by Theorem \ref{CWThm}, $L_G(\hat{d})$ is prime and $(D)$, being the ideal generated by the determinant of a generic matrix, is also prime. Additionally, $L_G(\hat{d})$ and $(D)$ use different variables, which implies that $L_G(\hat{d}) + (D) = \ini_w(L_G(d) + (D))$ is prime. Notice that this is possible because $d \geq \pmd(G) + k(G) +1$ and we can assume $t \leq k(G)$, where $D$ uses the last $t$ variables associated to the $t$ vertices neighbouring vertex $n$.
    We conclude that $L_G(d)+ (D)$ is prime on $S$ as wanted.\\

    For part \ref{PUFD} notice that from Lemma \ref{RDcong} we have 
    $$R_G(d)_D \cong R_{G'}(d)[x_1, \dots, x_{d-t}]_D.$$
    Then, as by induction hypothesis we have that $R_G'(d)$ was a unique factorization domain, we have that $R_G(d)_D$ is a unique factorization domain. By parts \ref{PPrime}, \ref{PUFD} and Theorem \ref{Nagata} we can conclude that $R_G(d)$ is a unique factorization domain.
\end{proof}

\section{Normality and Divisor Class Group}
\label{Sec4}

We have shown that when $d$ is large enough the LSS ring $R_G(d)$ has rational singularities and, therefore, it is normal. In this section we give a direct prove of normality in the case $G$ is a forest. Moreover, we compute the divisor class group of the ring $R_G(d)$ for two families of trees: the \textit{star} graph and the \textit{path} graph. Notice that these computations show that for a forest $G$ the bound for $R_G(d)$ being a UFD presented in Theorem \ref{ThmUFD} is sharp. We start this section by proving the normality.

\begin{theorem}
\label{ThmNormal}
    Let $G = (V, E)$ be a forest and $d \geq \Delta(G)+1$. Then, the LSS ring $R_G(d)$ is normal.
\end{theorem}
\begin{proof}
Notice that it is enough to focus on connected components of $G$ as the LSS-rings on each component are completely independent. Then, we consider the case in which $G$ is a tree. Also notice that by virtue of Theorem \ref{ThmUFD} it is enough to focus on the case $d = \Delta(G)+1$.

We use induction on $n$, the number of vertices of the tree $G$. If $n=1$, we have that $R_G(d) = S$ is a polynomial ring and then it is normal. Let $n>1$ and suppose $R_{G'}(d)$ is normal for $G'$ the subgraph of $G$ obtained by removing the leaf $k$ from $G$. Let us show, then, that $R_G(d)$ is normal.
By Serre's criterion for normality \ref{SerreNormal} we have to prove that the ring $R_G(d)$ has the properties $R_1$ and $S_2$. By Theorem \ref{CWThm} we have that as $d = \Delta(G)+1$, $L_G(d)$ is a complete intersection, hence $R_G(d)$ is Cohen-Macaulay and has the property $S_2$. Then, it is enough to show that $R_G(d)$ has the property $R_1$. To do so, take $P$ a prime ideal of $R_G(d)$ such that $\hgt P =1$ and let $h$ be the only neighbor of $k$ in $G$. We want to show that the ideal $P R_G(d)_P$ is a principal ideal. We consider two cases:
    \begin{enumerate}
        \item [\textbf{Case 1:}] There is a variable associated to vertex $h$, say $y_{hd}$, such that $y_{hd} \notin P$. Then, notice that by Lemma \ref{RDcong} we have
        $$R_G(d)_{y_{hd}} \cong R_{G'}(d)[y_{kj} \hspace{2pt} | \hspace{2pt} j\in [d-1]][y_{hd}^{-1}].$$
        The ring $R_{G'}(d)$ is normal by inductive hypothesis, then, being a polynomial extension of a normal ring, we have that $R_G(d)_{y_{hd}}$ is also normal. Localizing at $P$ we can conclude that $PR_G(d)_P$ is principal as wanted.
        \item [\textbf{Case 2:}] All variables associated to $h$ are in $P$.
        Let $I_h = (y_{h1}, \dots, y_{vh}) \subseteq P$ be the ideal generated by all the variables associated to vertex $h$. Notice that as 
        $$\Large\sfrac{R_G(d)}{I_h} \cong R_{G-\{h\}}(d),$$
        we have that $I_h$ is a prime ideal of height $\hgt (I_h) = d - \delta(h)$. Therefore, $\hgt (I_h) = 1$ and $\Delta(G) = \delta(h)$, i.e. the vertex $h$ has maximal degree in $G$.
        Then, we find another leaf $k'$ with neighbor $h'$ in $G$ such that there is a variable associated to it not in $P$ and return to Case 1.\\
        If there is no such $k'$, we can conclude that $G$ is a $(n-1)$-star graph.
        Without loss of generality, assume that the vertex $1$ is the central vertex and that $P$ is generated by the variables of associated to it. Then, $d = \Delta(G) + 1 = n$ and $L_G(d)$ has $(n-1)$ generators of shape
        $$f_{1j} = y_{11}y_{j1} + y_{12}y_{j2} + \dots + y_{1n}y_{jn},$$
        for $j = 2, \dots, n$.
        Also, $I_1 = (y_{11}, y_{12}, \dots, y_{1n}) = P$. We have to then prove that $PR_G(d)_{P}$ is principal. Writing the generators of the ideal $L_G(d)$ in matrix form, we have that in $R_G(d)$
            \begin{equation*}
                \begin{pmatrix}
      {y_{21}} & {y_{22}} & \dots  & {y_{2n}}\\
      {y_{31}} & {y_{32}} & \dots  & {y_{3n}}\\
      \vdots & & \ddots & \\
      {y_{n1}} & {y_{n2}} & \dots  & {y_{nn}}\\
     \end{pmatrix}
     \begin{pmatrix}
      {y_{11}} \\
      {y_{12}} \\
      \vdots \\
      {y_{1n}} \\
     \end{pmatrix} =0.
            \end{equation*}
            Then, by moving the last column of the matrix to the right hand side, we get
            \begin{equation}
            \label{eqY}
                \begin{pmatrix}
      {y_{21}} & {y_{22}} & \dots  & {y_{2(n-1)}}\\
      {y_{31}} & {y_{32}} & \dots  & {y_{3(n-1)}}\\
      \vdots & & \ddots & \\
      {y_{n1}} & {y_{n2}} & \dots  & {y_{n(n-1)}}\\
     \end{pmatrix}
     \begin{pmatrix}
      {y_{11}} \\
      {y_{12}} \\
      \vdots \\
      {y_{1(n-1)}} \\
     \end{pmatrix}
     = -
     \begin{pmatrix}
      {y_{2n}} \\
      {y_{3n}} \\
      \vdots \\
      {y_{nn}} \\
     \end{pmatrix} y_{1n}.
            \end{equation}
                Denote by $Y$ the first matrix in equation \ref{eqY} and notice that $Y$ is a $(n-1)$-squared matrix. If we multiply both sides of equation \ref{eqY} by $\adj(Y)$ we obtain that 
                \begin{equation*}
                    \det(Y) \begin{pmatrix}
      {y_{11}} \\
      {y_{12}} \\
      \vdots \\
      {y_{1(n-1)}} \\
     \end{pmatrix} \in (y_{1n}).
                \end{equation*}
                Notice that $Y$ does not use any of the variables generating $P$, then $\det(Y) \notin P$. Hence, when localizing at $P$ we have that $\det(Y)$ is invertible and therefore $y_{11}, y_{12}, \dots, y_{1(n-1)} \in (y_{1n})$. We can then conclude that $PR_G(d)_P$ is principal.

    \end{enumerate}

\end{proof}

Next, we recall some known results which are useful for the computation of divisor class groups. Theorem \ref{NagataThm} and Corollary \ref{Nagata2} are due to Nagata. For a complete proof of them the reader is referred to \cite{fossum2012divisor}.

\begin{theorem}
\label{NagataThm}
Let $R$ be a normal domain, and $S$ a multiplicatively closed subset of $R$. There exists the exact sequence of groups:

$$ 0 \quad \rightarrow \quad U \rightarrow \quad \Cl(R) \quad \xrightarrow{\hat{g}} \quad \Cl(R_S) \quad \rightarrow \quad 0,$$
where the map $\hat{g}$ sends $\cl(I)$ to $\cl(IR_S)$ and $U$ is the subgroup of $\Cl(R)$ generated by the classes of height $1$ prime ideals $P$ of $R$ such that $P \cap S \neq \emptyset.$
\end{theorem}

\begin{corollary}[\cite{ConcaThesis}]
\label{Nagata2}
Let $R$ be a normal domain and $B$ a factorial subring of $R$. Suppose there is an element $x \in B$ such that $B[x^{-1}] = R[x^{-1}].$ Let $x = x_1 \dots x_l$ be the factorization of $x$ in $B$ as a product of irreducible elements of $B$. Denote $P_1, \dots, P_r$ the minimal prime ideals of $(x)$ in $R$. Then $\Cl(R)$ is generated by $\cl(P_1), \dots, \cl(P_r)$. Furthermore, the syzygies between the given generators of $\Cl(R)$ are linear combinations of the syzygies
$$\sum \limits_{i=1}^{r} v_{P_i}(x_k)\cl(P_i) = 0, \quad \text{ for } k = 1, \dots, l.$$
\end{corollary}

We also recall a well known property of ideals next.

\begin{proposition}[Modular Law]
\label{ModLaw}
    Let $I,J,K$ be ideals of a ring $R$. If $J \subseteq I$, then $I \cap (J + K) = J + (I\cap K)$.
\end{proposition}

\subsection{$(n-1)$-star}
\label{SubSecS}
Let $G$ be the $(n-1)$-star graph, that is, the graph with $n$ vertices and $n-1$ edges, all of them adjacent to the same vertex, say vertex $n$. We have already mentioned that for a tree $\pmd(G)= \Delta(G)$ and $k(G) = 1$. Hence, Theorem \ref{ThmUFD} implies that $R_G(d)$ is UFD for $d \geq \Delta(G) +2$. When $d = \Delta(G) +1 = n$ we know by Theorem \ref{ThmNormal} that $R_G(n)$ is normal. Our goal for this subsection is to compute its divisor class group.

\begin{proposition}
\label{ThmDCGStar}
    Let $G = ([n], E)$ be the $(n-1)$-star graph. If $d = n$ then the divisor class group of $R_G(n)$ is
    $$\Cl(R_G(n)) \cong \mathbb{Z}.$$  
\end{proposition}
\begin{proof}
\label{PfPropStar}
Consider the $(n-1)$-star graph $G$ with central node $n$ associated to the $d=n$ variables $x_1, \dots, x_n$ and leaves $1, \dots, (n-1)$ associated to the variables $w_{1j}, \dots, w_{(n-1)j}$ respectively for $j \in \{1, \dots, n\}$. 
For $i= 1, \dots, (n-1)$, the polynomials
$$ f_i = \sum \limits_{j=1}^n w_{ij} x_j$$
define the  LSS -ideal $L_G(n) = (f_1, \dots, f_{(n-1)})$ on the polynomial ring
$$S= \mathbb{K}\left[x_1, \dots, x_n, w_{ij} \hspace{2pt} | \hspace{2pt} i \in [n-1], j \in [n]\right].$$
Let $R_G(n)$ be its respective LSS-ring. Let $G'$ be the subgraph of $G$ obtained by removing the leaf $k$ from $G$. Notice that $\pmd(G') = \pmd(G) - 1 = n-2$ and $k(G') = 1$, therefore by Theorem \ref{ThmUFD} we have that $R_G'(n)$ is a unique factorization domain. Using Lemma \ref{RDcong} we have that 
$$R_G(n)_{x_n} \cong R_{G'}(n)[w_{kj} \hspace{2pt} | \hspace{2pt} j\in [n-1]]_{x_n}$$ and therefore $R_G(n)_{x_n}$ is a UFD. This means, 
$$\Cl(R_G(n)_{x_n}) = 0.$$

Using Nagata's Corollary (\ref{Nagata2}) we have that
$$\Cl(R_G(n)) = \langle [P] \hspace{1pt} | \hspace{1pt} P\in \Min((\overline{x_n}))\rangle,$$
therefore we must calculate the minimal primes of $(L_G(n) + (x_n))$ in $S$.

Notice that $L_G(n) + (x_n) = L_G(n-1) + (x_n)$ in $S$. Since $L_G(n-1)$ does not involve the variable $x_n$ we may as well identify the minimal primes of $L_G(n-1)$.

Let $P$ be a prime ideal of $S$ such that $L_G(n-1) \subseteq P$. Let $W = (w_{ij})$ be the matrix which entries are the variables $w_{ij}$ with $i= 1, \dots, (n-1)$ and $ j = 1, \dots, (n-1)$ and let $\mathbf{x}$ be the column vector of variables $x_i$ with $i= 1, \dots, (n-1)$. We can write $L_G(n-1) = I_1(W\mathbf{x})$ where $I_1(W\mathbf{})$ is the ideal generated by the $1$-minors of $W\mathbf{x}$. Then we have $W\mathbf{x} = 0$ modulo $L_G(n-1)$. Multiplying both sides by $\adj W$ we obtain $(\det W)\mathbf{x} = 0$. This is $(det W)\mathbf{x} \in L_G(n-1) \subseteq P$. As $P$ is a prime ideal we have that either $det W \in P$ or $x \in P$. This leads us to make the following claim:

\textbf{Claim:} The primary decomposition of $L_G(n-1)$ is
$$L_G(n-1) = (\mathbf{x}) \cap (\det W, L_G(n-1)).$$

Notice that $L_G(n-1) \subseteq (\mathbf{x})$.
Let us prove the claim.

Using the Modular Law (\ref{ModLaw}), we have
$$(\mathbf{x}) \cap (\det W, L_G(n-1)) = L_G(n-1) + ((\mathbf{x}) \cap (\det W)).$$
Then, as $(\mathbf{x})$ and $(\det W)$ use different variables we have that $((\mathbf{x}) \cap (\det W)) = ((\mathbf{x})\cdot(\det W)) \subseteq L_G(n-1)$. Therefore we get
$$(\mathbf{x}) \cap (\det W, L_G(n-1)) = L_G(n-1).$$

Now, we show that both ideals are prime. Clearly $(\mathbf{x})$ is a prime ideal as it is generated by independent variables. For $\mathcal{G} = (\det W, L_G(n-1))$, notice that it is exactly the ideal presented by Herzog in Example $4$ of \cite{Herzog_1974}. In this example, Herzog proves indeed that $\mathcal{G}$ is a prime ideal.
Then we have $\Min(L_G(n-1)) = \{(\mathbf{x}), (\det W, L_G(n-1))\}$ which means
$$\Min (L_G(n) + (x_n)) = \Min(L_G(n-1) + (x_n)) = \{(\mathbf{x}, x_n), (\det W, L_G(n-1), x_n)\}.$$

Let us define the ideals  of $R_G(n)$,
\begin{align*}
    P_1 &= \Large\sfrac{(\mathbf{x}, x_n)}{L_G(n)},\\
    P_2 &= \Large\sfrac{(\det W, L_G(n-1), x_n)}{L_G(n)}.
\end{align*}
We have then that $\Min(\overline{x_n}) = \{P_1, P_2\}$ and hence
$$\Cl(R_G(n)) = \langle [P_1], [P_2] \rangle.$$

From Corollary\ref{Nagata2} we have that the syzygies between the generators are linear combinations of 
$$ v_{P_1}(x_n) [P_1] + v_{P_2}(x_n)[P_2] = 0,$$
where $v_{P_1}$ and $v_{P_2}$ are the valuations on the discrete valuation domains $R_G(n)_{P_1}$ and $R_G(n)_{P_2}$ respectively. Notice that $P_1R_G(n)_{P_1}$ is principal, then we can write $P_1R_G(n)_{P_1} = (y)$ for an element $y \in R_G(n)_{P_1}$.
Then as $(x_n) = P_1 \cap P_2$ is radical in $R_G(n)$, its localization in $P_1$ is given by
$$(x_n)R_G(n)_{P_1} = P_1R_G(n)_{P_1} \cap P_2R_G(n)_{P_1} = P_1R_G(n)_{P_1}$$
as $P_iR_G(n)_{P_j} = R_G(n)_{P_j}$ for all $i\neq j$. Then, taking $y = x_n$ we have that the valuation $v_{P_1}(x_n) = 1$. Analogously we have that $v_{P_2}(x_n) = 1$.
We can then conclude that
$$\Cl(R_G(n)) = \Large \sfrac{\mathbb{Z}^2}{(1,1)} \cong \mathbb{Z}.$$

\end{proof}

\subsection{$n$-path}
\label{SubSecP}

In this subsection we will calculate the divisor class group of the ring $R_G(d)$ associated to another family of graphs known as the $n$-path graphs (or linear graphs). As the name suggests, an $n$-path graph is a connected graph with $n$ vertices which can be listed in the order $1, \dots, n$ in such a way that the edges are $\{i, i+1\}$ for $i = 1, \dots, n-1$. Notice that for $n\geq 3$, the $n$-path $G$ has $\Delta(G) = 2$, then to calculate the divisor class group of $R_G(d)$ we will fix $d = \Delta(G) + 1 = 3$.
In an attempt to make the document more readable and comprehensive we will first calculate the divisor class group for a $4$-path. This example exhibits the main ideas behind the proof of Proposition \ref{ThmDCGPath}. We will first state a useful proposition and then exhibit the $4$-path example.

\begin{proposition}
\label{PropMultip}
    Let $R$ be a graded ring, $I$ a homogeneous ideal of $R$ and $P_i \in \Min(I)$ for $i= 1, \dots, n$ and $P_i \neq P_j$ for $i\neq j$. If $\hgt P = \hgt I$ for all $P \in \Ass\left(\large \sfrac{R}{I}\right)$ and
    $$\sum\limits_{i=1}^n e \left( \large \sfrac{R}{P_i} \right) \geq e\left(\large\sfrac{R}{I}\right)$$
    then, $I$ is radical and $\Min (I) = \{P_1, \dots, P_n\}$.
\end{proposition}

\begin{proof}
    Let $I = Q_1 \cap Q_2 \cap \dots \cap Q_r$ be a primary decomposition of $I$. Then, without loss of generality, we can assume that $\sqrt{Q_i} = P_i$ for $i=1, \dots, n$. Let us also call $\sqrt{Q_i} = P_i$ for $i = n+1, \dots, r$.
    In general we know that 
    \begin{equation}
    \label{BrunsHerzog478}
        e\left( \large \sfrac{R}{I} \right) = \sum_{P} l \left( \left( \large \sfrac{R}{I}\right)_{P}\right) e \left( \large\sfrac{R}{P} \right),
    \end{equation}
    where the sum is taken over all prime ideals $P \in \Ass(\sfrac{R}{I})$ such that $\hgt P = \hgt I$. For a proof of equation \ref{BrunsHerzog478} the reader is addressed to \cite[Corollary 4.7.8.]{BrunsHerzog98}.
    As $\hgt P = \hgt I$ for all $P \in \Ass (\large \sfrac{R}{I})$ we have that
    $$e\left( \large \sfrac{R}{I} \right) = \sum_{i=1}^r l \left( \left( \large \sfrac{R}{I}\right)_{P_i}\right) e \left( \large\sfrac{R}{P_i} \right).$$
    Finally, as $\sum\limits_{i=1}^n e \left( \large \sfrac{R}{P_i} \right) \geq e\left(\large\sfrac{R}{I}\right)$, we can conclude that $r=n$ which means that $I$ has no other minimal primes. Also, $l \left( \left( \large \sfrac{R}{I}\right)_{P_i}\right) =1 $ for all $i=1, \dots , n$, which implies that $I$ is radical.
\end{proof}

\begin{example}
\label{ex1}
    Let $G$ be the linear graph with $4$ vertices and $3$ edges and $d = 3$. Notice that by Theorem \ref{CWThm}, $L_G(3) = (f_{12}, f_{23}, f_{34})$ is a prime radical complete intersection, but $R_G(3)$ it is not necessarily a unique factorization domain. Then, it makes sense to calculate its divisor class group. Let $\mathbf{x} = y_{23}y_{32}$, we will localize $R_G(3)$ at $\mathbf{x}$. Notice that $y_{23}$ and $y_{32}$ are variables associated to the vertices of maximal degree in $G$.
    
    Let $H$ be the graph consisting in just one edge $\{1,2\}$. Notice that $\pmd(H) = \Delta(H) = 1$ and $k(H)=1$, therefore by Theorem \ref{ThmUFD} we have that $R_H(3)$ is a unique factorization domain. Then, using iteratively Lemma \ref{RDcong} with $t=1$ we get that
    $$R_G(3)_{\mathbf{x}} \cong R_H(3)_\mathbf{x}[y_{31}, y_{33}, y_{41}, y_{42}].$$
    Hence, $R_G(3)_{\mathbf{x}}$ is a UFD. This means that $\Cl(R_G(3)_{\mathbf{x}}) = 0$ and using Nagata's Corollary (\ref{Nagata2}) we have that $\Cl(R_G(3)) = \langle [P] \hspace{1pt} | \hspace{1pt} P\in \Min((\overline{\mathbf{x}}))\rangle$.

    \textbf{Claim:} the set of minimal primes of $(\mathbf{x}) \subseteq R_G(3)$ is $\Min(\mathbf{x}) = \{P_2, P_3, Q_2, Q_3\}$, where
    \begin{align*}
        P_2 &= (y_{21}, y_{22}, y_{23}, f_{34}),\\
        P_3 &= (y_{31}, y_{32}, y_{33}, f_{12}),\\
        Q_2 &= \left( I_1\left(\begin{pmatrix}
        y_{11} & y_{12} \\ y_{31} & y_{32}
    \end{pmatrix} \begin{pmatrix}
        y_{21} \\ y_{22}
    \end{pmatrix} \right),  \det \begin{pmatrix}
        y_{11} & y_{12} \\ y_{31} & y_{32}
    \end{pmatrix}, y_{23}, f_{34} \right),\\
       Q_3 &= \left( I_1\left(\begin{pmatrix}
        y_{21} & y_{23} \\ y_{41} & y_{43}
    \end{pmatrix} \begin{pmatrix}
        y_{31} \\ y_{33}
    \end{pmatrix} \right),  \det \begin{pmatrix}
        y_{21} & y_{23} \\ y_{41} & y_{43}
    \end{pmatrix}, y_{32}, f_{12}\right).
    \end{align*}

We will divide the proof of the claim into two parts. In the first part we will show that $P_2, P_3, Q_2$ and $Q_3$ are prime ideals containing $(\mathbf{x})$. In the second part we will show that they are the only minimal primes of $(\mathbf{x})$.

\begin{enumerate}
    \item [Part I:]
It is clear that $P_2, P_3, Q_2, Q_3$ contain $(\mathbf{x})$, let us show that they are prime.
It is easy to see that $P_2$ and $P_3$ are prime as they are generated by unrelated variables and a polynomial using another set of variables.

Let us now show that $Q_2$ is prime. Let $A = \mathbb{K}[y_{11}, y_{12}, y_{21}, y_{22}, y_{23}, y_{31}, y_{32}]$, $B = A[y_{41}, y_{42}, y_{43}]$ and 
 $$J_0 = \left( I_1\left(\begin{pmatrix}
        y_{11} & y_{12} \\ y_{31} & y_{32}
    \end{pmatrix} \begin{pmatrix}
        y_{21} \\ y_{22}
    \end{pmatrix} \right) ,  \det \begin{pmatrix}
        y_{11} & y_{12} \\ y_{31} & y_{32}
    \end{pmatrix}, y_{23} \right) \subseteq A.$$
    
Notice that $f_{34}$ is a linear polynomial in $B$ with coefficients in $A/J_0$. Let $M$ be the $\Large \sfrac{A}{J_0}$-module with finite free resolution
\begin{center}
    \begin{tikzcd}
0 \arrow[r] & \Large\sfrac{A}{J_0} \arrow[r, "C"] & \left(\Large\sfrac{A}{J_0}\right)^3 \arrow[r] & M \arrow[r] & 0,
\end{tikzcd}
\end{center}
where $C= \begin{pmatrix}  y_{31}\\y_{32}\\y_{33} \end{pmatrix}$. Then, as $Q_2 = J_0 + f_{34}$ we have that $Sym_{\sfrac{A}{J_0}}(M) =\Large\sfrac{B}{Q_2}$.

Notice that $J_0 \subseteq A$ is prime as it is generated by the generators of the \textit{Herzog ideal}, seen in Subsection \ref{SubSecS} and \cite{Herzog_1974}, and a variable unrelated to such generators. Therefore, $A/J_0$ is a Noetherian domain. Also, $A/J_0$ is Cohen-Macaulay, so it satisfies the property $S_n$ for all $n$. Using \textit{Huneke's criterion} \cite{Huneke_1981} we have that $Sym_{A/J_0}(M)$ is a domain if and only if $\grade(I_1(c)) \geq 2$. By Cohen-Macaulayness, this translates to $Q_2$ being prime if and only if $\hgt((y_{31}, y_{32}, y_{33})) \geq 2$ in $A/J_0$. Notice that $y_{33}$ is not used in $J_0$, then $\hgt((y_{33})) = 1$. Also, there is no relation between $y_{31}, y_{32}$ and $y_{33}$ in $J_0$. We can conclude that the height of $(y_{31}, y_{32}, y_{33})$ in $A$ is at least $2$. Therefore, $Q_2$ is prime.
Showing that $Q_3$ is prime is analogous.

\item[Part II:]
We want to show that $\Min(\mathbf{x}) = \{P_2, P_3, Q_2, Q_3\}$. By Proposition \ref{PropMultip} we have to show that
\begin{enumerate}
    \item $\hgt Q = \hgt (\mathbf{x})$ for all $Q \in \Ass\left(\large \sfrac{R_G(3)}{(\mathbf{x})}\right)$
    \item $P_2, P_3, Q_2, Q_3 \in \Min(\mathbf{x})$
    \item $e\left(\large\sfrac{R_G(3)}{P_2}\right) + e\left(\large\sfrac{R_G(3)}{P_3}\right) + e\left(\large\sfrac{R_G(3)}{Q_2}\right) + e\left(\large\sfrac{R_G(3)}{Q_3}\right) \geq e\left(\large\sfrac{R_G(3)}{\mathbf{x}R_G(3)}\right)$
\end{enumerate}
To show part (a) recall that $R_G(3)$ is a Cohen-Macaulay domain. Then, as $\mathbf{x}$ is a non zero divisor we have that $\large\sfrac{R_G(3)}{(\mathbf{x})}$ is also Cohen-Macaulay. Hence the \textit{Unmixedness Theorem} \cite[Thm. 2.1.6.]{BrunsHerzog98} holds which means that as $\hgt (\mathbf{x}) = 1$, the ideal $(\mathbf{x})$ is unmixed and we can conclude part (a).

For part (b) we calculate the heights of $P_2, P_3, Q_2, Q_3$ in $R_G(3)$.
Notice that 
\begin{align*}
    \hgt \left( \large\sfrac{P_2}{L_G(3)}\right) &= \dim \left( \large\sfrac{S}{L_G(3)}\right) - \dim \left(\large\sfrac{S}{P_2}\right)\\
    &= (\#\text{variables in S}-\#E) - (\#\text{variables in S}-\#\text{generators of }P_2)\\
    &= \#\text{generators of }P_2 - \#E = 1.
\end{align*}
Equivalently we have that $\hgt \left( \large \sfrac{P_3}{L_G(3)}\right) = 1$.
On the other hand, Herzog \cite{Herzog_1974} proved that the ideal
\begin{equation}
    \label{G2}
    \mathcal{G}_2 = \left( I_1\left(\begin{pmatrix}
        y_{11} & y_{12} \\ y_{31} & y_{32}
    \end{pmatrix} \begin{pmatrix}
        y_{21} \\ y_{22}
    \end{pmatrix} \right) ,  \det \begin{pmatrix}
        y_{11} & y_{12} \\ y_{31} & y_{32}
    \end{pmatrix} \right)
\end{equation}

has height $2$ in $S$.

Let us define a weight vector $\omega$ in the following way
\begin{equation*}
  \omega(y_{ij}) = \begin{cases}
    1 \text{ for } i=3, j=3\\
    0 \text{ otherwise. }
\end{cases}  
\end{equation*}
Then $\ini_{\omega}(f_{34}) = y_{33}y_{44}$, while the other generators of $Q_2$ remain unchanged under taking the initial form with respect to $\omega$. Notice that $\ini_{\omega}(f_{34}), \ini_{\omega}(y_{23})$ is a regular sequence modulo $\mathcal{G}_2$, then by Proposition \ref{GBPropRS2} we have that $f_{34}, y_{23}$ is a regular sequence modulo $\mathcal{G}_2$.
Therefore $\hgt Q_2 = 4$ and we have that\begin{align*}
    \hgt \left( \large\sfrac{Q_2}{L_G(3)}\right) &= \dim \left( \large\sfrac{S}{L_G(3)}\right) - \dim \left(\large\sfrac{S}{Q_2}\right)\\
    &= (\#\text{variables in S}-\#E) - (\#\text{variables in S}-\hgt Q_2)\\
    &= \hgt Q_2 - \#E = 1.
\end{align*}
Analogously we have $\hgt \left( \large\sfrac{Q_3}{L_G(3)}\right)=1$. Hence in $R_G(3)$ we have
$$ht(P_2) = ht(P_3) = ht(Q_2) = ht(Q_3) = 1.$$
As $\hgt (\mathbf{x}) = 1$, we conclude that $P_2, P_3, Q_2, Q_3$ are minimal prime ideals containing the ideal $(\mathbf{x})$, that is $P_2, P_3, Q_2, Q_3 \in \Min (\mathbf{x})$.

For part (c) recall that $L_G(3)$ is a complete intersection generated by three polynomials of degree $2$, therefore
$$e\left(R_G(3)\right) = 2^3 = 8.$$
As $\mathbf{x}$ is a non zero divisor of $R_G(3)$ of degree $2$, we have
$$e\left(\large\sfrac{R_G(3)}{\mathbf{x}}\right) = 2 \cdot 8 = 16.$$
Notice that $P_2$ and $P_3$ are also complete intersection ideals with three degree $1$ generators and one degree $2$ generator. Therefore,
$$e\left(\large\sfrac{R_G(3)}{P_2}\right) = e\left(\large\sfrac{R_G(3)}{P_3}\right) = 2.$$

For calculating the multiplicity of $Q_2$ and $Q_3$, notice that for a generic $2\times 2$ matrix of variables $A$ and a generic $2$-vector $\mathbf{z}$ of variables not in $A$, the \textit{Herzog-type ideal} $\mathcal{G} = \left( I_1(A \mathbf{z}), \det A \right)$ coincides with the determinantal ideal $I_2$ of a generic $3\times 2$ matrix of variables. Therefore, it is not hard to see that in this case $\mathcal{G}_2$ has multiplicity $3$. A general formula for calculating the multiplicity for determinantal ideals is presented in \cite{HERZOG92}.
Now, notice that $y_{23},f_{34}$ is a regular sequence modulo $\mathcal{G}_2$ hence we have that
$$e\left(\large\sfrac{R_G(3)}{Q_2}\right) = 3 \cdot 2 = 6.$$
Analogously,
$$e\left(\large\sfrac{R_G(3)}{Q_3}\right) = 6.$$
Then,
\begin{align*}
    e\left(\large\sfrac{R_G(3)}{\mathbf{x}R_G(3)}\right) &= 16\\
    &= e\left(\large\sfrac{R_G(3)}{P_2}\right) + e\left(\large\sfrac{R_G(3)}{P_3}\right) + e\left(\large\sfrac{R_G(3)}{Q_2}\right) + e\left(\large\sfrac{R_G(3)}{Q_3}\right).
\end{align*}
Therefore, our claim is true.
\end{enumerate}

So far, by Nagata's Corollary \ref{Nagata2}, we have found the generators for $\Cl(R_G(3))$. In order to figure out the relations between such generators we have to calculate the valuations of $y_{23}$ and $y_{32}$ in the localization of $R_G(3)$ at each of the minimal primes.
Notice that as $(\mathbf{x}) = (y_{23}y_{32})$ is a radical ideal we have, 

\begin{minipage}[position]{5cm}
\begin{align*}
    (\mathbf{x})R_G(3)_{P_2} &= P_2R_G(3)_{P_2}\\
    (\mathbf{x})R_G(3)_{P_3} &= P_3R_G(3)_{P_3}
\end{align*}
 \end{minipage}
\begin{minipage}[position]{5cm}
\begin{align*}
    (\mathbf{x})R_G(3)_{Q_2} &= Q_2R_G(3)_{Q_2},\\
    (\mathbf{x})R_G(3)_{Q_3} &= Q_3R_G(3)_{Q_3},
\end{align*}
 \end{minipage}
 
\vspace{5pt}

Hence, $P_iR_G(3)_{P_i}$ and $Q_iR_G(3)_{Q_i}$ are principal ideals generated by $\mathbf{x}$ for $i= 2, 3$.
Let us find the valuation of $y_{23}$ in $R_G(3)_{P_2}$. We want to find a unit $\alpha$ of $R_G(3)_{P_2}$ and an integer $k$ such that
$$y_{23} = \alpha (y_{23}y_{32})^k.$$
As $y_{32} \notin P_2$, we have that $(y_{32})^{-1}$ is a unit in $R_G(3)_{P_2}$ and
$$y_{23} = (y_{32})^{-1}(y_{23}y_{32})^1,$$
which implies $v_{P_2}(y_{23}) = 1$. Analogously, as $y_{32}\notin Q_2$, we have that $v_{Q_2}(y_{23}) = 1$. On the other hand, notice that $y_{23} \notin P_3$ and $y_{23} \notin Q_3$, therefore they are units in $R_G(3)_{P_3}$ and $R_G(3)_{Q_3}$ respectively, which means that $v_{P_3}(y_{23}) = 0 = v_{Q_3}(y_{23})$.
An analogous argument leads us to conclude that $v_{P_2}(y_{32}) = 0 = v_{Q_2}(y_{32})$ and $v_{P_3}(y_{32}) = 1 = v_{Q_3}(y_{32})$.
The syzygies between the generators of $\Cl(R_G(3))$ are, then, linear combinations of

\begin{align*}
    [P_2]+[Q_2]=&0\\
    [P_3]+[Q_3]=&0.
\end{align*}
This means that $\Cl(R_G(3)) \cong \mathbb{Z}^2$.
\end{example}

\vspace{7pt}

In order to generalize the discussion exposed in Example \ref{ex1} we will specify some notation and include some useful propositions.

For each $i = 2, \dots, n-1$, define the list
\begin{equation}
\label{Fi}
    F_i = f_{12}, \dots, \hat{f}_{(i-1)i}, \hat{f}_{i(i+1)}, \dots, f_{(n-2)(n-1)}
\end{equation}
of all polynomials generating $L_G(3)$ besides $f_{(i-1)i}$ and $f_{i(i+1)}$. Consider the matrix
$$Y_i = \begin{pmatrix}
        y_{(i-1)1} & y_{(i-1)k} \\ y_{(i+1)1} & y_{(i+1)k}
    \end{pmatrix}$$
 and the vector $y_i = \begin{pmatrix}
        y_{i1} \\ y_{ik}
    \end{pmatrix}$
 where $k = 2$ for $i$ even and $k=3$ for $i$ odd.
 We define also the ideals $P_i$ and $Q_i$ for $i = 2, \dots, n-1$ in the following way,
 
\begin{align}
\label{defPQ}
    P_i =& \left(y_{i1}, y_{i2}, y_{i3}, F_i\right) \\
    Q_i =& \left( I_1\left(Y_iy_i \right),  det(Y_i), y_{i\hat{k}}, F_i \right),
\end{align}
where $\hat{k} = \begin{cases}
    2 \text{ for } k=3\\
    3 \text{ for } k=2.
\end{cases}$

Let
\begin{equation}
\label{x}
\mathbf{x} = y_{23}y_{32}\dots y_{(n-1)k},
\end{equation}
where $k = 2$ for $n$ even and $k=3$ for $n$ odd.

\begin{proposition}
\label{Rmk1} 
    Let $G$ be an $n$-path, $S$ be the polynomial ring in $3n$ variables and $R_G(3)$ the LSS-ring associated to $G$. The ideals $P_i$ and $Q_i$ defined in \ref{defPQ} are such that $\hgt(P_i) = \hgt(Q_i) = 1$ in $R_G(3)$ for all $i=2, \dots, n-1$.
\end{proposition}

\begin{proof}
    Notice that by Cohen-Macaulayness we have that
    \begin{equation}
        \label{eq4}
        \hgt \left(\large \sfrac{P_i}{L_G(3)}\right) = \dim \left(\large\sfrac{S}{L_G(3)}\right) - \dim \left(\large\sfrac{S}{P_i}\right).
    \end{equation}
    Then, we have that using Theorem \ref{CWThm}, as $d= 3\geq \pmd(G)$, $L_G(3)$ is a complete intersection. Notice that $P_i$ is generated by a regular sequence as none of the polynomials in $F_i$ use the variables $y_{i1}, y_{i2}, y_{i3}$. Hence, we can write equation (\ref{eq4}) as
    \begin{align*}
         \hgt \left(\large \sfrac{P_i}{L_G(3)}\right) &= (\#\text{variables in }S - \#E )-(\#\text{variables in }S - \#\text{generators of } P_i)\\
         &= \#\text{generators of } P_i-\#E \\
         &= n- (n-1) = 1.
    \end{align*}
    To calculate the height of $Q_i$ recall that Herzog proved in \cite{Herzog_1974} that the ideal
    $$\mathcal{G}_i = \left( I_1\left(Y_iy_i \right),  \det(Y_i)\right)$$
    has height $2$ for any $i \in \mathbb{N}$. As $y_{ik}$ is not used in the other generators of $Q_i$, to prove that $Q_i$ has the desired height it is enough to prove that $F_i$ is a regular sequence modulo $\mathcal{G}_i$. To do so we show that for each $i$ it is possible to define a weight vector $\omega$ such that $\ini_{\omega}(\mathcal{G}_i) = \mathcal{G}_i$ and $\ini_{\omega}(f_{rs})$ form a regular sequence modulo $\mathcal{G}_i$ for $f_{rs} \in F_i$. We exhibit the construction of the weight vector for $i=4$ in the $7$-path, its generalization follows immediately.
    
    Consider the $7$-path graph and the ideal
    $$Q_4 = \left( I_1\left(\begin{pmatrix}
        y_{31} & y_{32} \\ y_{51} & y_{52}
    \end{pmatrix} \begin{pmatrix}
        y_{41} \\ y_{42}
    \end{pmatrix} \right) ,  \det \begin{pmatrix}
        y_{31} & y_{32} \\ y_{51} & y_{52}
    \end{pmatrix} , f_{12}, f_{23}, f_{56}, f_{67}, y_{43}\right).$$
    Define the weight vector in the following way

    \begin{equation*}
        \omega(y_{ij}) =\begin{cases}
        1 \quad  \text{ for } y_{12}, y_{33}, y_{53}, y_{72}\\
        0 \quad \text{ otherwise.}
    \end{cases}
    \end{equation*}
    Notice that with this weight vector we have

    \begin{minipage}{5.5cm}
    \begin{align*}
        \ini_{\omega}(f_{12}) &= y_{12}y_{22},\\
        \ini_{\omega}(f_{23}) &= y_{23}y_{33},\\
    \end{align*}
    \end{minipage}
    \begin{minipage}{5.5cm}
    \begin{align*}
        \ini_{\omega}(f_{56}) &= y_{53}y_{63},\\
        \ini_{\omega}(f_{67}) &= y_{62}y_{72}.\\
    \end{align*}
    \end{minipage}

    Then, as each initial form uses different variables we have that the initial forms $\ini_{\omega}(f_{12}), \ini_{\omega}(f_{23}), \ini_{\omega}(f_{56}), \ini_{\omega}(f_{67})$ form a regular sequence in $\large\sfrac{S}{\mathcal{G}_4}$. Using Proposition \ref{GBPropRS2} we have that $f_{12}, f_{23}, f_{56}, f_{67}$ form a regular sequence modulo $\mathcal{G}_4$ as wanted.

    Extending this procedure for any $n$-path and any $Q_i$ we can conclude that $F_i$ is a regular sequence modulo $\mathcal{G}_i$. Then,
    $\hgt Q_i = \hgt \mathcal{G}_i + \# F_i + 1 = n.$
    We then have that
    \begin{align*}
        \hgt \left( \large \sfrac{Q_i}{L_G(3)}\right) &= \dim \left( \large\sfrac{S}{L_G(3)}\right)- \dim \left(\large\sfrac{S}{Q_i}\right)\\
        &= (\# \text{variables in } S - \#E) - (\# \text{variables in } S - \hgt Q_i)\\
        &= \hgt Q_i - \#E = 1.
    \end{align*}
    Hence, $\hgt P_i = \hgt Q_i = 1$ in $R_G(3)$.
\end{proof}

\begin{lemma}
    Let $R_G(3)$ be the LSS-ring of an $n$-path graph. Consider $P_i$ and $Q_i$ stated as in \ref{defPQ} for $i = 2, \dots, n-1$ and let $\mathbf{x}$ be as stated in \ref{x}.
    We have that
    $$e\left(\large \sfrac{R_G(3)}{(\mathbf{x})}\right) = \sum \limits_{i=1}^{n-2} e\left(\large\sfrac{R_G(3)}{P_i}\right) + \sum \limits_{i=1}^{n-2} e\left(\large\sfrac{R_G(3)}{Q_i}\right).$$
    \label{lemmamultip}
\end{lemma}

\begin{proof}
    Recall $L_G(3)$ is a complete intersection by Theorem \ref{CWThm}. Therefore, $f_{12}, \dots, f_{(n-1)n}$ is a regular sequence of $(n-1)$ polynomials of degree $2$. Hence,
    $$e(R_G(3)) = 2^{(n-1)}.$$
    As $\mathbf{x}$ is a non zero divisor of $R_G(3)$ of degree $n-2$, we have that 
    $$e \left( \large\sfrac{R_G(3)}{(\mathbf{x})} \right) = (n-2) 2^{(n-1)}.$$
    On the other side we have that $P_i$ is generated by distinct variables and a regular sequence of $n-3$ polynomials of degree $2$ on different variables then,
    $$e \left( \large\sfrac{R_G(3)}{P_i} \right) = 2^{(n-3)}.$$

    The ideal $Q_i$ is generated by a \textit{Herzog type ideal} $\mathcal{G}$ of size $2$, $(n-3)$ polynomials of degree $2$ which form a regular sequence modulo $\mathcal{G}$ and an extra variable. 
    As we saw in Example \ref{ex1} the multiplicity of $\mathcal{G}$ is $3$, hence the multiplicity of $Q_i$ is given by
    $$e(\large\sfrac{R_G(3)}{Q_i}) = 3 \cdot 2^{n-3}.$$
    We can conclude that
    \begin{align*}
        e \left( \large\sfrac{R_G(3)}{(\mathbf{x})} \right) &= (n-2) 2^{n-1} \\
        &= (n-2) \left( 2^{n-3} + 3 \cdot 2^{n-3} \right)\\
        &= \sum \limits_{i=1}^{n-2} e(\large\sfrac{R_G(3)}{P_i}) + \sum \limits_{i=1}^{n-2} e(\large\sfrac{R_G(3)}{Q_i}).
    \end{align*}  
\end{proof}

In the following theorem we generalize the claim of Example \ref{ex1}. The arguments for proving it are analogous except for an extra step that is required when proving that $Q_i$ is prime when $n \geq 5$.

\begin{proposition}
    \label{PMinPrimes}
    Let $G$ be an $n$-path graph, $R_G(3)$ the LSS-ring associated to $G$ with $d=3$ and $\mathbf{x}$ as defined in \ref{x}.
    The set of minimal primes of the ideal $(\mathbf{x}) \subseteq R_G(3)$ is
    $$\Min(\mathbf{x}) = \{P_2, \dots, P_{n-1}, Q_2, \dots, Q_{n-1}\},$$
    with $P_i$ and $Q_i$ described in \ref{defPQ}.
\end{proposition}

\begin{proof}
We first show that $P_i, Q_i \in \Min(\mathbf{x})$.
Notice that the polynomials in $F_i$ defined at \ref{Fi} are the generators of a  LSS -ideal of two disconnected path graphs with less vertices than $n$. Therefore, by Theorem \ref{CWThm} we have that $F_i$ generates a prime ideal. The ideal $P_i$ is, then, generated by the prime ideal $(F_i)$ and the variables $y_{i1}, \dots, y_{i3}$ which do not appear in any of the polynomials in $F_i$. Therefore, $P_i$ is a prime ideal for $i = 2, \dots, n-1$.
Now, to show that $Q_i$ is a prime ideal we will proceed by induction on the number $n$ of vertices of $G$. In Example \ref{ex1} we proved the result for $n=4$. Supposing the result true for $n-1$, we will show that it is also true for $n$.

Let $Q_{i,j}$ denote the ideal of type $Q_i$ for the $j$-path graph and let
$$I = (y_{(n-1)1}, y_{(n-1)2}, y_{(n-1)3}) \subseteq R_G(3).$$

We use \textit{Huneke's criterion} \cite{Huneke_1981} on $Q_{i,j}$. That is, we need to show that $\hgt I \geq 2$ in $\large\sfrac{S}{Q_{i,n-1}}$. As $R_G(3)$ is a Cohen-Macaulay ring we have that
$$\hgt \left(\large\sfrac{I}{Q_{i,n-1}}\right) =  \dim \left(\large\sfrac{S}{Q_{i,n-1}}\right) - \dim \left(\large\sfrac{S}{I}\right).$$
Notice that $Q_{i,n-1} = (Q_{i,n-2},f_{(n-2)(n-1)})$, then
$$\dim\left(\large\sfrac{S}{Q_{i,n-1}}\right) \geq \dim\left(\large\sfrac{S}{Q_{i,n-2}}\right) - 1$$
and we have
$$\hgt\left(\large\sfrac{I}{Q_{i,n-1}}\right) \geq  \dim \left(\large\sfrac{S}{Q_{i,n-2}}\right) - \dim \left(\large\sfrac{S}{I}\right) -1.$$
Then, again by Cohen-Macaulayness, we have that 
$$\dim \left(\large\sfrac{S}{Q_{i,n-2}}\right) - \dim \left(\large\sfrac{S}{I}\right) = \hgt\left(\large\sfrac{I}{Q_{i,n-2}}\right)$$
and we get
$$\hgt\left(\large\sfrac{I}{Q_{i,n-1}}\right) \geq \hgt\left(\large\sfrac{I}{Q_{i,n-2}}\right) - 1.$$
Notice that as none of the variables that generate $I$ are used in $Q_{i,n-2}$, the height of $I$ in $\large\sfrac{S}{Q_{i,n-2}}$ is $3$. Hence $\hgt\left(\large\sfrac{I}{Q_{i,n-1}}\right) \geq 2$ as we wanted, which means that $Q_i = Q_{i,n}$ is prime.

We have that $P_i$ and $Q_i$ are prime ideals containing $(\mathbf{x})$ for $i = 2, \dots, n-1$. From Proposition \ref{Rmk1} we also have that $\hgt P_i = \hgt Q_i = 1$ and recall that $\hgt (\mathbf{x}) = 1$. Then, $P_i, Q_i \in \Min (\mathbf{x})$ for $i= 2, \dots, n-1$.

Recall that by Theorem \ref{CWThm} the ring $R_G(3)$ is a Cohen-Macaulay domain, then the Unmixed Theorem holds. In particular, as $(\mathbf{x})$ is a principal ideal and $\hgt (\mathbf{x}) =1$ we have that $(\mathbf{x})$ is unmixed. That is $\hgt (\mathbf{x}) = \hgt(Q)$ for all $Q \in \Ass \left( \large \sfrac{R_G(3)}{(\mathbf{x})}\right)$.
Also, recall that by Lemma \ref{lemmamultip} we have
$$e\left(\large \sfrac{R_G(3)}{(\mathbf{x})}\right) = \sum \limits_{i=1}^{n-2} e\left(\large\sfrac{R_G(3)}{P_i}\right) + \sum \limits_{i=1}^{n-2} e\left(\large\sfrac{R_G(3)}{Q_i}\right).$$

Using Proposition \ref{PropMultip} we conclude that
$$\Min(\mathbf{x}) = \{P_2, \dots, P_{n-1}, Q_2, \dots, Q_{n-1}\}.$$
\end{proof}

We can now calculate the divisor class group for an $n$-path obtaining the main result of this subsection.

\begin{proposition}
\label{ThmDCGPath}
    Let $G = ([n], E)$ be an $n$-path graph. If $d = 3$ then the divisor class group of $R_G(3)$ is
    $$\Cl(R_G(3)) \cong \mathbb{Z}^{n-2}.$$
\end{proposition}

\begin{proof}
    Recall that $L_G(3)$ is a prime radical complete intersection ideal by Theorem \ref{CWThm} and $R_G(d)$ is normal by Theorem \ref{ThmNormal}. To calculate its divisor class group we will localize $R_G(3)$ at the variables associated to the vertices of degree $2$. That is, we localize at $\mathbf{x} \in R_G(3)$ as defined in equation \ref{x}.
    Let $H$ be the graph consisting in just one edge $\{1,2\}$. Notice that $\pmd(H) = 1$ and $k(H) = 1$, then by Theorem \ref{ThmUFD} we know that $R_H(3)$ is a unique factorization domain. Then, applying repeatedly Lemma \ref{RDcong} we have that
    $$R_G(3)_{\mathbf{x}} \cong R_H(3)[y_{i1}, y_{i2} \hspace{2pt} | \hspace{2pt} i \in \{3, \dots, n\}, i \text{ even}][y_{i1}, y_{i3} \hspace{2pt} | \hspace{2pt} i \in \{3, \dots, n\}, i \text{ odd}]_{\mathbf{x}}.$$
    Then, $R_G(3)_{\mathbf{x}}$ is a UFD and $\Cl(R_G(3)_{\mathbf{x}}) = 0$. Using Corollary \ref{Nagata2}, we have that
    $$\Cl(R_G(3)) = \langle [P] \hspace{1pt}|\hspace{1pt} P \in \Min(\mathbf{x})\rangle,$$
    and by Proposition \ref{PMinPrimes} we can conclude that
     $$\Cl(R_G(3)) = \langle [P_i], [Q_i] \hspace{2pt}|\hspace{2pt} i= 2, \dots, n-1\rangle,$$
    with $P_i, Q_i$ defined as in \ref{defPQ}.
    
    To understand the relations between the generators of $\Cl(R_G(3))$ first notice that as $(\mathbf{x})$ is a radical ideal. We have that
    \begin{align*}
        (\mathbf{x})R_G(3)_{P_i} &= P_iR_G(3)_{P_i}\\
        (\mathbf{x})R_G(3)_{Q_i} &= Q_iR_G(3)_{Q_i}.
    \end{align*}
    Then, $P_i$ and $Q_i$ are principal ideals generated by $\mathbf{x}$ in $R_G(3)_{P_i}$ and $R_G(3)_{Q_i}$ respectively. Let
    $$\mathbf{x_i} = y_{23}y_{32}\dots y_{(i-1)l}y_{(i+1)l} \dots y_{(n-1)k},$$
    for $l=2$ or $l=3$ depending on the parity of $i$. Notice that $\mathbf{x_i} \notin P_i$, therefore it is a unit in $R_G(3)_{P_i}$. For $y_{il} \in P_i$ we have that $y_{il} = (\mathbf{x_i})^{-1} (\mathbf{x})^1.$ Hence, the valuation of $y_{il}$ in $R_G(3)_{P_i}$ is $v_{P_i}(y_{il}) = 1$.
    Analogously, we have that $\mathbf{x_i} \notin Q_i$ and it is a unit in $R_G(3)_{Q_i}$. We can conclude that $v_{Q_i}(y_{il}) = 1$.
    On the other side we have that for $j \neq i$ we have that $y_{jl}\notin P_i$ and $y_{jl}\notin Q_i$. This means that $y_{jl}$ is a unit in $R_G(3)_{Pi}$ and $R_G(3)_{Q_i}$ and its respective valuations are $v_{P_i}(y_{il}) = 0$ and $v_{Q_i}(y_{il}) = 0$.
    Hence, the syzygies between the generators of $\Cl(R_G(3))$ are
    $$[P_i] + [Q_i]= 0$$
    for $i= 2, \dots, n-1$.
   This means that there are $2(n-2)$ generators and $n-2$ relations and we can conclude that
    $$\Cl(R_G(3)) \cong \mathbb{Z}^{(n-2)}.$$    
\end{proof}

\bibliographystyle{plain}
\bibliography{biblio} 

\end{document}